\documentclass[11pt]{article}
\usepackage{amsmath}
\usepackage{amsfonts}
\usepackage{amssymb}

\usepackage{ntheorem,lipsum}
\theorembodyfont{\upshape}
\newtheorem{theorem}{Theorem}

\newtheorem{definition}{Definition}

\newtheorem{lemma}{Lemma}
\newtheorem{proposition}{Proposition}
\newtheorem{remark}{Remark}
\newenvironment{proof}[1][Proof]{\noindent\textbf{#1.} }{\ \rule{0.5em}{0.5em}}

\begin{document}
\title{\textbf{Continuous quaternion Stockwell transform\\ and Uncertainty principle}}
\author{ Brahim Kamel
\thanks{  Department of Mathematics, College of Science, University of Bisha, Bisha, Saudi Arabia.\newline E-mail:
kbrahim@ub.edu.sa}
\thanks{ Faculty of Sciences of Tunis. University of Tunis El Manar, Tunis, Tunisia.\newline  E-mail :
kamel.brahim@fst.utm.tn}
\qquad \& \quad Emna Tefjeni
\thanks{ Faculty of Sciences of Tunis. University of Tunis El Manar, Tunis, Tunisia.\newline  E-mail :
tefjeni.emna@outlook.fr}}
\date{}
\maketitle

\begin{abstract}In this paper we propose a novel transform called continuous quaternionic stockwell transform. We express the admissibility condition in term of the (two-sided) quaternion Fourier transform . We show that its fundamental properties, such as Plancherel, Parseval and inversion formula, can be established whenever the continuous quaternion stockwell satisfy a particular admissibility condition. We present several examples of the continuous quaternion stockwell transform.
We apply the continuous quaternion stockwell transform properties and the two sided quaternion Fourier transform to establish a number of uncertainty principle for these extended Stockwell .
\end{abstract}
\vspace{2mm}
\section{Introduction:}
The quaternion Fourier transform , which is a nontrivial generalization of the real and complex Fourier transform using quaternion algebra has been of interest to researchers for some years
. It was found that many Fourier transform  properties still hold but others have to be modified.
On the other hand, R.G.Stockwell constructed the Stockwell transform by combining of wavelet transform and windowed Fourier transform \cite{stockwellclassic} . Recently, some authors \cite{akila} demonstrated a number of properties of the classical stockwell transform considered the kernel of the classical Fourier transform and using quaternion valued function .
The purpose of this article is to construct the 2D  continuous  quaternion stockwell transform based on quaternion algebra. Our construction uses the kernel of the quaternion Fourier transform which in general does not commute with quaternions. We use the two sided quaternion Fourier transform to investigate some important properties of the 2D continuous  quaternion stockwell transform. Special attention is devoted to Plancherel, Parseval's relation, inversion formula. We show that these fundamental properties can be established whenever the admissible quaternion window satisfy a particular admissibility condition. Using the properties of the 2D continuous  quaternion stockwell transform and the uncertainty principle for the two sided quaternion Fourier transform, we establish an uncertainty principle for the 2D  continuous  quaternion stockwell transform.
\section{Generalities:}
In this section, we recall some basic definitions and properties of the Quaternion Fourier transform.\\
For more details, see \cite{quaternion}.\\
The quaternion algebra was formally introduced by the Irish mathematician W.R Hamilton in 1843, and it is a generalization of complex numbers. The quaternion algebra over $\mathbb{R}$, denoted by $\mathbb{H}$, is an associative non-commutative four-dimensional algebra,
\begin{center}
$\mathbb{H} = \{q = q_{r} + iq_{i} + jq_{j} + kq_{k} \hspace{0.1 cm}|\hspace{0.1 cm} q_{r}, q_{i}, q_{j}, q_{k} \in \mathbb{R}\},$
\end{center}
which obey Hamilton's multiplication rules
\begin{center}
$ij = - ji = k$,\hspace{0.2 cm} $jk = -kj = i$,\hspace{0.2 cm}  $ki = -ik = j$,\hspace{0.2 cm} $i^{2} = j^{2} = k^{2} = ijk = -1$.
\end{center}
The quaternion conjugate of a quaternion $q$ is given by
\begin{center}
$\overline{q} = q_{r} - iq_{i} - jq_{j} - kq_{k}$, \hspace{0.3 cm} $ q_{r}, q_{i}, q_{j}, q_{k} \in \mathbb{R}$.
\end{center}
The quaternion conjugation is a linear anti-involution
\begin{center}
$\overline{pq} = \overline{q}\ \overline{p}$, \hspace{0.3 cm} $\overline{p+q} = \overline{p}+\overline{q}$, \hspace{0.3 cm} $\overline{\overline{p}} = p$.
\end{center}
The modulus of a quaternion $q$ is defined by
\begin{center}
$|q| = \sqrt{q\overline{q}} = \sqrt{q_{r}^{2}+ q_{i}^{2}+q_{j}^{2}+ q_{k}^{2}}$.
\end{center}
It is not difficult to see that
\begin{center}
$|pq| = |p| |q|$, \hspace{0.3 cm} $\forall p,q \in\mathbb{H}$.
\end{center}
A quaternion-valued function $f:\mathbb{R}^{d} \longrightarrow \mathbb{H}$ will be written as \begin{center}
$f(x) = f_{r}(x) + i f_{i}(x) + jf_{j}(x) +kf_{k}(x)$
\end{center}
with real-valued coefficient functions $f_{r}, f_{i}, f_{j}, f_{k}:\mathbb{R}^{d}\longmapsto\mathbb{R}$.\\
We introduce the space $L^{2}(\mathbb{R}^{d},\mathbb{H})$ as the left module of all quaternion-valued function \\ $f:\mathbb{R}^{d}\longmapsto\mathbb{H}$ with finite norm
\begin{center}
$ \|f\|_{2,\mathbb{R}^{d}} = \bigg(\displaystyle\int_{\mathbb{R}^{d}} |f(x)|^{2} d\mu_{d}(x)\bigg)^{\frac{1}{2}}$.
\end{center}
We defined by $\mu_{d}$ the normalized Lebesgue measure on $\mathbb{R}^{d}$ by $d\mu_{d}(x) = \dfrac{dx}{(2\pi)^{\frac{d}{2}}} $ .\\
If $1\leq p < \infty$, the $L^{p}$-norm of $f$ is defined by
\begin{equation}\label{normp}
\|f\|_{p,\mathbb{R}^{d}} = \bigg(\displaystyle\int_{\mathbb{R}^{d}} |f(x)|^{p} d\mu_{d}(x)\bigg)^{\frac{1}{p}}.
\end{equation}
For $p = \infty$, $L^{\infty}(\mathbb{R}^{d},\mathbb{H})$ is a collection of essentially bounded measurable functions with the norm
$$\|f\|_{\infty,\mathbb{R}^{d}} = ess \displaystyle\sup_{x\in \mathbb{R}^{d}} |f(x)|$$
if $f \in L^{\infty}(\mathbb{R}^{d},\mathbb{H})$ is continuous then
\begin{equation}\label{norminfty}
\|f\|_{\infty,\mathbb{R}^{d}} = \displaystyle\sup_{x\in \mathbb{R}^{d}} |f(x)|.
\end{equation}
For $p=2$, we can define the quaternion-valued inner product
\begin{equation}\label{&}
\langle f,g\rangle_{2,\mathbb{R}^{d}} = \displaystyle\int_{\mathbb{R}^{d}} f(x) \overline{g(x)}\ d\mu_{d}(x)
\end{equation}
with symmetric real scalar part
\begin{equation}\label{pro}
(f,g)_{2,\mathbb{R}^{d}}\  = \ \dfrac{1}{2}\big[\langle f,g\rangle_{2,\mathbb{R}^{d}} + \langle g,f\rangle_{2,\mathbb{R}^{d}}\big]\ = \ \displaystyle\int_{\mathbb{R}^{d}} Sc(f(x) \overline{g(x)})\  d\mu_{d}(x) = \ \ Sc\bigg(\displaystyle\int_{\mathbb{R}^{d}} f(x) \overline{g(x)}\  d\mu_{d}(x)\bigg).
\end{equation}
Both (\ref{&}) and (\ref{pro}) lead to the $L^{2}(\mathbb{R}^{d},\mathbb{H})$-norm
\begin{equation}
\|f\|_{2,\mathbb{R}^{d}} = \sqrt{(f,f)_{2,\mathbb{R}^{d}}} = \sqrt{\langle f,f\rangle_{2,\mathbb{R}^{d}}}  = \bigg(\displaystyle\int_{\mathbb{R}^{d}} |f(x)|^{2}  d\mu_{d}(x)\bigg)^{\frac{1}{2}}.
\end{equation}
As a consequence of the inner product (\ref{&}) we obtain the quaternion Cauchy-Schwartz inequality
\begin{equation}\label{cs}
\bigg|\displaystyle\int_{\mathbb{R}^{d}} f(x)\overline{g(x)} \  d\mu_{d}(x)\bigg| \leq \bigg(\displaystyle\int_{\mathbb{R}^{d}} |f(x)|^{2}\  d\mu_{d}(x)\bigg)^{\frac{1}{2}}\bigg(\displaystyle\int_{\mathbb{R}^{d}} |g(x)|^{2}\  d\mu_{d}(x)\bigg)^{\frac{1}{2}}\hspace{0.2 cm} \forall f , g \in L^{2}(\mathbb{R}^{d},\mathbb{H}).
\end{equation}
For two function $f$, $g$ $\in L^{2}(\mathbb{R}^{d},\mathbb{H})$. Using (\ref{&}), (\ref{pro}) and (\ref{cs}) the Schwartz inequality takes the form
\begin{equation}\label{hes1}
\bigg(\displaystyle\int_{\mathbb{R}^{d}}(g(x)\overline{f(x)} + f(x)\overline{g(x)})d\mu_{2d}(x)\bigg)^{2} \leq 4 \bigg(\displaystyle\int_{\mathbb{R}^{2d}}|f(x)|^{2}d\mu_{2d}(x)\bigg) \bigg(\displaystyle\int_{\mathbb{R}^{2d}}|g(x)|^{2}d\mu_{2d}(x)\bigg).
\end{equation}

\begin{definition} \cite{14}
\mbox{}\\
The Quaternion Fourier transform of a function $f\in L^1(\mathbb{R}^{2},\mathbb{H})\cap L^2(\mathbb{R}^{2},\mathbb{H})$ is defined as
$$\mathcal{F}_{Q}(f)(u,v) := \int_{\mathbb{R}^{2}}e^{-ix.u} f(x,y) e^{-jy.v} d\mu_{2}(x,y).$$
and it satisfies Plancherel's formula $\|\mathcal{F}_{Q}(f)\|_{2,\mathbb{R}^{2}}=\|f\|_{2,\mathbb{R}^{2}}.$ As a consequence $\mathcal{F}_{Q}$ extends to a
unitary operator on $L^2(\mathbb{R}^{2},\mathbb{H})$ and satisfies Parseval's formula:
$$( f, g )_{2,\mathbb{R}^{2}} = ( \mathcal{F}_{Q}(f), \mathcal{F}_{Q}(g) )_{2,\mathbb{R}^{2}},\ \ \ \forall \ f,\ g\in L^2(\mathbb{R}^{2},\mathbb{H}).$$
The inverse Quaternion Fourier transform of a function $f\in L^1(\mathbb{R}^{2},\mathbb{H})$ is given as
$$\mathcal{F}_{Q}^{-1}( f )(x,y) = \mathcal{F}_{Q}( f )(-x,-y).$$
Thus, if $f \in L^1(\mathbb{R}^{2},\mathbb{H})$ with $\mathcal{F}_{Q}(f)\in L^1(\mathbb{R}^{2},\mathbb{H})$, then
$$f(x,y)=\int_{\mathbb{R}^{2}} e^{i x.u} \mathcal{F}_{Q}(f)(u,v)\ e^{j y.v} d\mu_{2}(u,v).$$
\end{definition}
\begin{definition}
The convolution of two functions $f,\ g \in L^2(\mathbb{R}^{2},\mathbb{H})$ is the function $f*g$ defined by
$$(f * g)(x)=\int_{\mathbb{R}^{2}}f(t)g(x-t)d\mu_{2}(t) = \int_{\mathbb{R}^{2}}f(x-t)g(t)d\mu_{2}(t),\ \ \ x\in \mathbb{R}^{2}.$$
\end{definition}
\begin{theorem}\cite{article3} \label{conditions}
\mbox{}\\
Let $f$ and $g$ be two quaternion functions and if we assume that
\begin{eqnarray}\label{hyp}
\mathcal{F}_{Q}(g)(u,v) e^{-j v.y} = e^{-j v.y} \mathcal{F}_{Q}(g)(u,v) \ \ and \ \ \mathcal{F}_{Q}(jg)(u,v) e^{-j v.y} =j  e^{-j v.y} \mathcal{F}_{Q}(g)(u,v)
\end{eqnarray}
then the quaternion Fourier transform of the convolution of $f\in L^{2}(\mathbb{R}^{2},\mathbb{H})$ and $g\in L^{2}(\mathbb{R}^{2},\mathbb{H})$ is given as
\begin{eqnarray}\label{conv}
\mathcal{F}_{Q}(f*g)(u,v) =  \mathcal{F}_{Q}(f)(u,v)\mathcal{F}_{Q}(g)(u,v)
\end{eqnarray}
\end{theorem}
\begin{remark}
\mbox{}
\begin{enumerate}
\item
Let $g(x,y) = g_{1}(x,y) + j g_{2}(x,y)$  where $g_{1}$ and $g_{2}$ in $L^{2}(\mathbb{R}^{2},\mathbb{R})$ and $g(-x, y) = g(x, y)$ we have $g$ satisfies the condition (\ref{hyp}).
\item
 In \cite{bahri}, the authors gave a demonstration of the property by considering some other assumptions.
\end{enumerate}
\end{remark}
For all $f \in L^2(\mathbb{R}^{2},\mathbb{H})$  and $g$ satisfies the condition of theorem \ref{conditions}, we have
$$f * g=\mathcal{F}_{Q}^{-1}( \mathcal{F}_{Q}(f).\mathcal{F}_{Q}(g) ).$$
Thus, if $f \in L^2(\mathbb{R}^{2},\mathbb{H})$  and $g$ satisfies the condition of theorem \ref{conditions}, the function $f * g$ belongs to $L^2(\mathbb{R}^{2},\mathbb{H})$ if and only if
$\mathcal{F}_{Q}(f).\mathcal{F}_{Q}(g)$ belongs to $L^2(\mathbb{R}^{2},\mathbb{H})$, and in this case, we have
$$\mathcal{F}_{Q}(f * g)=\mathcal{F}_{Q}(f).\mathcal{F}_{Q}(g).$$
Then, for all $f \in L^2(\mathbb{R}^{2},\mathbb{H})$  and $g$ satisfies the condition of theorem \ref{conditions}, we have
\begin{eqnarray}\label{theorem}
\int_{\mathbb{R}^{2}}|f * g(x,y)|^2 d\mu_{2}(x,y)=\int_{\mathbb{R}^{2}}|\mathcal{F}_{Q}(f)(u,v)|^2|\mathcal{F}_{Q}(g)(u,v)|^2 d\mu_{2}(u,v),
\end{eqnarray}
where both sides are finite or both sides are infinite.
\section{Continuous quaternion stockwell transform}
Based on the properties of quaternions and the definition of the classical Stockwell transform $\mathcal{S}_{\varphi}$ with the Fourier transform $\mathcal{F}$, we obtain the definition of the quaternion Stockwell transform $\mathcal{S}^{Q}_{\varphi}$ by replacing the kernel of $\mathcal{F}$ with the kernel of the quaternion Fourier transform $\mathcal{F}_{Q}$ in the classical definition  $\mathcal{S}_{\varphi}$ as follows. \\
In this section, we present the Continuous quaternion stockwell transform and we establish some new results (Parseval formula, inversion formula,Lieb inequality,..). For more details on stockwell transform, the reader can see \cite{book,riba}.\\
Let $I_{2}$ denote the $(2,2)$-identity matrix and $0_{2}$, resp. $I_{2}$ the vectors with $2$ entries $0$, resp. 1.
For all $\xi = (\xi_{1},\xi_{2}) \in \mathbb{R}^{2}$, let $A_{\xi}$ be an invertible $2\times 2$ matrix, i.e.,
\begin{equation}
A_\xi :=\left(
   \begin{array}{ccc}
     \xi_{1} &  &0\\
     \\
     0 &  & \xi_{2}
   \end{array}
 \right).
\end{equation}
A pair $\{\varphi,\psi\}$ of a nonzero functions in $L^{1}(\mathbb{R}^{2},\mathbb{H})\cap L^{2}(\mathbb{R}^{2},\mathbb{H})$ is said to be an admissible quaternion window pair if $\varphi$ and $\psi$ satisfy the following admissibility condition
\begin{equation}
C_{\{\varphi,\psi\}} = \displaystyle\int_{\mathbb{R}^{2}} \mathcal{F}_{Q}\overline{(\varphi)}(1-\xi) \overline{\mathcal{F}_{Q}\overline{(\psi)}(1-\xi)} \dfrac{d\mu_{2}(\xi)}{|\mbox{det}A_{\xi}|}
\end{equation}
is a non zero quaternion constant .\\
A nonzero function $\varphi$ in  $L^{1}(\mathbb{R}^{2},\mathbb{H})\cap L^{2}(\mathbb{R}^{2},\mathbb{H})$ is said to be an admissible quaternion window if $\varphi$ satisfies the following admissibility condition
\begin{equation}
C_{\varphi} = \displaystyle\int_{\mathbb{R}^{2}} |\mathcal{F}_{Q}\overline{(\varphi)}(1-\xi)|^{2} \dfrac{d\mu_{2}(\xi)}{|\mbox{det}A_{\xi}|}
\end{equation}
is a non zero real positive constant .\\
                         Let $\varphi$ in $L^{1}(\mathbb{R}^{2},\mathbb{H})\cap L^{2}(\mathbb{R}^{2},\mathbb{H})$ be such that
\begin{equation}
\displaystyle\int_{\mathbb{R}^{2}} \varphi(x) \ d\mu_{2}(x) = 1,
\end{equation}
then the two-dimensional quaternion Stockwell transform $S_{\varphi}^{Q}(f)$ of a quaternion signal $f$ in $L^{2}(\mathbb{R}^{2},\mathbb{H})$ with respect to quaternion window $\varphi$ is given by
\begin{equation}
(S_{\varphi}^{Q}f)(\xi,b) = |\mbox{det}\ A_{\xi}|^{\frac{1}{2}} \displaystyle\int_{\mathbb{R}^{2}}  e^{-ix_{1}\xi_{1}} f(x) e^{-jx_{2}\xi_{2}} \overline{\varphi(A_{\xi}(x-b))}  d\mu_{2}(x).\\
\end{equation}
for all $b$ in $\mathbb{R}^{2}$ and all $\xi$ in $\mathbb{R}^{2}$.\\
Let $\varphi$ in $L^{1}(\mathbb{R}^{2},\mathbb{H})\cap L^{2}(\mathbb{R}^{2},\mathbb{H})$ be a non zero quaternion window function. Then the inversion formula for the two-dimensional quaternion Stockwell transform $S_{\varphi}^{Q}$ of a quaternion signal $f$ in $L^{2}(\mathbb{R}^{2},\mathbb{H})$ holds. Indeed, for all  $\xi \in \mathbb{R}^{2}\setminus\{0\}$,
\ \\ \ \\
$\displaystyle\int_{\mathbb{R}^{2}}(S_{\varphi}^{Q}f)(b,\xi)\ \dfrac{d\mu_{2}(b)}{|\mbox{det}\ A_{\xi}|^{-\frac{1}{2}}}$
\begin{eqnarray*}
& = & |\mbox{det}\ A_{\xi}|\displaystyle\int_{\mathbb{R}^{2}} e^{-ix_{1}\xi_{1}} f(x) e^{-jx_{2}\xi_{2}} \bigg(\displaystyle\int_{\mathbb{R}^{2}} \overline{\varphi(A_{\xi}(x-b))}\ d\mu_{2}(b)\bigg)\  d\mu_{2}(x) \\
& = & \displaystyle\int_{\mathbb{R}^{2}} e^{-ix_{1}\xi_{1}} f(x) e^{-jx_{2}\xi_{2}} \overline{\bigg(\displaystyle\int_{\mathbb{R}^{2}} \varphi(y)\ d\mu_{2}(y)\bigg)}\  d\mu_{2}(x) \\
& = & \displaystyle\int_{\mathbb{R}^{2}} e^{-ix_{1}\xi_{1}} f(x) e^{-jx_{2}\xi_{2}}\  d\mu_{2}(x) \\
& = & \mathcal{F}_{Q}(f)(\xi).
\end{eqnarray*}
Thus,
\begin{equation*}
f = \mathcal{F}_{Q}^{-1}AS_{\varphi}^{Q}f,
\end{equation*}
where $\mathcal{F}_{Q}^{-1}$ is the inverse quaternion Fourier transform and $A$ is the operator given by
\begin{equation*}
(AF)(\xi) = |\mbox{det}\ A_{\xi}|^{\frac{1}{2}} \ \displaystyle\int_{\mathbb{R}^{2}} F(\xi,b)\ d\mu_{2}(b)
\end{equation*}
for all $\xi$ in $\mathbb{R}^{2}\setminus\{0\}$ and all measurable functions $F$ on $\mathbb{R}^{2}\times\mathbb{R}^{2}$, provided that the integral exists.\\
We note also that the two-dimensional quaternion Stockwell transform $S_{\varphi}$ can be written as
\begin{equation}\label{stockwell}
(S_{\varphi}^{Q}f)(\xi,b) = \langle M_{\xi}f,D_{A_{\xi}}T_{-b}\varphi\rangle_{2,\mathbb{R}^{2}} = M_{\xi}f\divideontimes \overline{D_{A_{\xi}}\check{\varphi}}(b)
\end{equation}
where $M_{\xi}$, $T_{-b}$ and $D_{A_{\xi}}$ are the modulation operator, the translation operator and the dilation operator given, respectively, by
\begin{eqnarray*}
&(M_{\xi}f)(x) = e^{-ix_{1}\xi_{1}} f(x)\ e^{-jx_{2}\xi_{2}},&\\
&(T_{-b}f)(x) = f(x-b)&
\end{eqnarray*}
and
\begin{equation*}
(D_{A_{\xi}}f)(x) = |\mbox{det}A_{\xi}|^{\frac{1}{2}} f(A_{\xi}x)
\end{equation*}
for all $x$ in $\mathbb{R}^{2}$.

\begin{remark} Let $\xi = (\xi_{1},\xi_{2}) \in \mathbb{R}^{2}\setminus \{0\}$. Then, we have
\begin{enumerate}
\item
\begin{equation}
A_{\xi}^{T} = A_{\xi}.
\end{equation}
\item
\begin{equation}
A_{\xi}^{-1} :=\left(
   \begin{array}{ccc}
     \frac{1}{\xi_{1}} &  &0\\
     \\
     0 &  & \frac{1}{\xi_{2}}
   \end{array}
 \right).
\end{equation}
\item
\begin{equation}
\mbox{det}\ A_{\xi} = \xi_{1}.\xi_{2}\ \ \mbox{and} \ \ \mbox{det}\ A_{\xi}^{-1} = (\mbox{det}\ A_{\xi})^{-1}.
\end{equation}
\item
\begin{equation}
A_{\xi}\ x = A_{\xi}^{T}\ x = (\xi_{1} x_{1},\xi_{2} x_{2}).
\end{equation}
\item
\begin{equation}
A_{\xi}^{-1}\ x = \big(\frac{x_{1}}{\xi_{1}},\frac{x_{2}}{\xi_{2}}\big)
\end{equation}
\end{enumerate}
\end{remark}
\begin{proposition} For all $\eta$ in $\mathbb{R}^{2}$,
\begin{enumerate}
\item
\begin{equation}\label{prop1}
\mathcal{F}_{Q}(M_{\xi}f)(\eta) = (T_{\xi}\mathcal{F}_{Q}(f))(\eta),\ \ \ \xi \in \mathbb{R}^{2}.
\end{equation}
\item
\begin{equation}\label{prop2}
\mathcal{F}_{Q}(T_{-b}f)(\eta) = (M_{b}\mathcal{F}_{Q}(f))(\eta),\ \ \ b\in\mathbb{R}^{2}.
\end{equation}
\item
\begin{equation}\label{prop3}
\mathcal{F}_{Q}(D_{A_{\xi}}f)(\eta) = (D_{A_{\xi}^{-1}}\mathcal{F}_{Q}(f))(\eta),\ \ \ \xi\in \mathbb{R}^{2}\setminus\{0\}.
\end{equation}
\end{enumerate}
\end{proposition}
\begin{proof}
\begin{enumerate}
\item For all $\xi$ and $\eta$ in $\mathbb{R}^{2}$, the quaternion Fourier transform of $M_{\xi}f$ is given by
\begin{eqnarray*}
\mathcal{F}_{Q}(M_{\xi}f)(\eta)
& = & \displaystyle\int_{\mathbb{R}^{2}} e^{-ix_{1}\eta_{1}}M_{\xi}f(x) e^{-jx_{2}\eta_{2}} d\mu_{2}(x) \\
& = & \displaystyle\int_{\mathbb{R}^{2}} e^{-ix_{1}\eta_{1}} e^{-ix_{1}\xi_{1}} f(x) e^{-jx_{2}\xi_{2}} e^{-jx_{2}\eta_{2}} d\mu_{2}(x) \\
& = & \displaystyle\int_{\mathbb{R}^{2}} e^{-ix_{1}(\eta_{1}+\xi_{1})}  f(x) e^{-jx_{2}(\xi_{2}+\eta_{2})}  d\mu_{2}(x) \\
& = & \mathcal{F}_{Q}(f)(\eta + \xi) = (T_{\xi}\mathcal{F}_{Q}(f))(\eta),
\end{eqnarray*}
where $(\eta + \xi) = (\eta_{1}+\xi_{1},\eta_{2}+\xi_{2})$.
\item For all  $b$ and  $\eta$ in $\mathbb{R}^{2}$, the quaternion Fourier transform of $T_{-b}f$ is given by
\begin{eqnarray*}
\mathcal{F}_{Q}(T_{-b}f)(\eta)
& = & \displaystyle\int_{\mathbb{R}^{2}} e^{-ix_{1}\eta_{1}}T_{-b}f(x)\ e^{-jx_{2}\eta_{2}} d\mu_{2}(x) \\
& = & \displaystyle\int_{\mathbb{R}^{2}} e^{-ix_{1}\eta_{1}} f(x-b)\ e^{-jx_{2}\eta_{2}} d\mu_{2}(x) \\
& = & \displaystyle\int_{\mathbb{R}^{2}} e^{-i(y_{1}+b_{1})\eta_{1}} f(y)\ e^{-j(b_{2}+y_{2})\eta_{2}} d\mu_{2}(y) \\
& = & e^{-i b_{1}\eta_{1}} \mathcal{F}_{Q}(f)(\eta)\ e^{-j b_{2}\eta_{2}} \\
& = & (M_{b}\mathcal{F}_{Q}(f))(\eta).
\end{eqnarray*}
\item for all $\xi$ in $\mathbb{R}^{2}\setminus\{0\}$ and all $\eta$ in $\mathbb{R}^{2}$, the quaternion Fourier transform of $D_{A_{\xi}}f$ is given by
\begin{eqnarray*}
\mathcal{F}_{Q}(D_{A_{\xi}}f)(\eta)
& = & \displaystyle\int_{\mathbb{R}^{2}} e^{-ix_{1}\eta_{1}}D_{A_{\xi}}f(x)\ e^{-jx_{2}\eta_{2}} d\mu_{2}(x)\\
& = & |\mbox{det}A_{\xi}|^{\frac{1}{2}} \displaystyle\int_{\mathbb{R}^{2}} e^{-ix_{1}\eta_{1}}f(A_{\xi} x)\ e^{-jx_{2}\eta_{2}} d\mu_{2}(x)\\
& = & |\xi_{1}\xi_{2}|^{\frac{1}{2}} \displaystyle\int_{\mathbb{R}^{2}} e^{-ix_{1}\eta_{1}} f(\xi_{1}x_{1},\xi_{2}x_{2}) e^{-jx_{2}\eta_{2}} d\mu_{2}(x)\\
& = &  |\xi_{1}\xi_{2}|^{-\frac{1}{2}} \displaystyle\int_{\mathbb{R}^{2}} e^{-i\frac{\eta_{1}}{\xi_{1}}y_{1}} f(y) e^{-j\frac{\eta_{2}}{\xi_{2}}y_{2}} d\mu_{2}(y)\\
& = & |\xi_{1}\xi_{2}|^{-\frac{1}{2}} \mathcal{F}_{Q}(f)(\frac{\eta_{1}}{\xi_{1}},\frac{\eta_{1}}{\xi_{1}})\\
& = & |\mbox{det}A_{\xi}^{-1}|^{\frac{1}{2}} \mathcal{F}_{Q}(f)(A_{\xi}^{-1}\eta)\\
& = & (D_{A_{\xi}^{-1}}\mathcal{F}_{Q}(f))(\eta).
\end{eqnarray*}
\end{enumerate}

\end{proof}
\begin{theorem} (Basic properties of quaternion Stockwell transform)\\
Let $\varphi$, $\psi$ be two admissible quaternion windows in $L^{1}(\mathbb{R}^{2},\mathbb{H})\cap L^{2}(\mathbb{R}^{2},\mathbb{H})$. For every quaternion functions $f,g$ in $L^{2}(\mathbb{R}^{2},\mathbb{H})$, the quaternion Stockwell transform satisfies the following properties
\begin{itemize}
\item[(i)] \textbf{Linearity}
\begin{equation*}
(S_{\varphi}^{Q}(\alpha f + \beta g))(\xi,b) = \alpha (S_{\varphi}^{Q}f)(\xi,b) + \beta (S_{\varphi}^{Q}g)(\xi,b),\ \ \ \alpha,\beta \in \mathbb{C}.
\end{equation*}
\item[(ii)] \textbf{Anti-linearity}
\begin{equation*}
(S_{\alpha\varphi+\beta\psi}f)(\xi,b) = (S_{\varphi}f)(\xi,b)\overline{\alpha} + (S_{\psi}f)(\xi,b)\overline{\beta},\ \ \ \alpha, \beta \in \mathbb{H}.
\end{equation*}
\item[(iii)] \textbf{Parity}
\begin{equation*}
(S_{\varphi}^{Q}\check{f})(\xi,b) = (S_{\varphi}^{Q}f)(-\xi,-b),
\end{equation*}
where $\check{f}(x) = f(-x)$.
\item[(vi)] \textbf{Scaling}
\begin{equation}\label{scaling}
(S_{\varphi}^{Q}f_{\lambda})(\xi,b)  =  \dfrac{(S_{\varphi}^{Q}f)(\frac{\xi}{\lambda},\lambda b)}{|\lambda|},\ \ \ \lambda \in \mathbb{R}\setminus\{0\}.
\end{equation}
\end{itemize}
\end{theorem}
\begin{proof}
\begin{itemize}
\item[(i)] For any two complex scalars $\alpha$, $\beta$ in $\mathbb{C}$, we have
\begin{eqnarray*}
S_{\varphi}^{Q}(\alpha f + \beta g)(\xi,b)
& = & |\mbox{det}A_{\xi}|^{\frac{1}{2}} \displaystyle\int_{\mathbb{R}^{2}} e^{-ix_{1}\xi_{1}} (\alpha f(x)+\beta g(x)) e^{-jx_{2}\xi_{2}} \overline{\varphi(A_{\xi}(x-b))}d\mu_{2}(x)\\
& = & \alpha\bigg( |\mbox{det}A_{\xi}|^{\frac{1}{2}} \displaystyle\int_{\mathbb{R}^{2}} e^{-ix_{1}\xi_{1}} f(x) e^{-jx_{2}\xi_{2}} \overline{\varphi(A_{\xi}(x-b))}d\mu_{2}(x)\bigg)\\
& + & \beta\bigg( |\mbox{det}A_{\xi}|^{\frac{1}{2}} \displaystyle\int_{\mathbb{R}^{2}} e^{-ix_{1}\xi_{1}} g(x) e^{-jx_{2}\xi_{2}} \overline{\varphi(A_{\xi}(x-b))}d\mu_{2}(x)\bigg)
\\
& = & \alpha (S_{\varphi}^{Q}f)(\xi,b) + \beta (S_{\varphi}^{Q}g)(\xi,b).
\end{eqnarray*}
\item[(ii)] For any two quaternion scalars $\alpha, \beta$ in $\mathbb{H}$, we have
\begin{eqnarray*}
(S_{\alpha\varphi +\beta\psi}^{Q}f)(\xi,b)
& = &  |\mbox{det}A_{\xi}|^{\frac{1}{2}} \displaystyle\int_{\mathbb{R}^{2}} e^{-ix_{1}\xi_{1}} f(x) e^{-jx_{2}\xi_{2}} \overline{(\alpha\ \varphi(A_{\xi}(x-b)) + \beta\ \psi(A_{\xi}(x-b)))}d\mu_{2}(x)\\
& = & \bigg( |\mbox{det}A_{\xi}|^{\frac{1}{2}} \displaystyle\int_{\mathbb{R}^{2}} e^{-ix_{1}\xi_{1}} f(x) e^{-jx_{2}\xi_{2}} \overline{\varphi(A_{\xi}(x-b))}d\mu_{2}(x)\bigg) \  \overline{\alpha}  \\
& + & \bigg( |\mbox{det}A_{\xi}|^{\frac{1}{2}} \displaystyle\int_{\mathbb{R}^{2}} e^{-ix_{1}\xi_{1}} g(x) e^{-jx_{2}\xi_{2}} \overline{\varphi(A_{\xi}(x-b))}d\mu_{2}(x)\bigg) \ \overline{\beta}  \\
& = & (S_{\varphi}^{Q}f)(\xi,b)\ \overline{\alpha} + (S_{\varphi}^{Q}g)(\xi,b) \ \overline{\beta}.
\end{eqnarray*}
\item[(iii)]We have
\begin{eqnarray*}
(S_{\varphi}^{Q}\check{f})(\xi,b)
& = & |\mbox{det}A_{\xi}|^{\frac{1}{2}} \displaystyle\int_{\mathbb{R}^{2}} e^{-ix_{1}\xi_{1}} \check{f}(x) e^{-jx_{2}\xi_{2}} \overline{\varphi(A_{\xi}(x-b))}d\mu_{2}(x)\\
& = & |\mbox{det}A_{\xi}|^{\frac{1}{2}} \displaystyle\int_{\mathbb{R}^{2}} e^{-ix_{1}\xi_{1}} f(-x) e^{-jx_{2}\xi_{2}} \overline{\varphi(A_{\xi}(x-b))}d\mu_{2}(x)\\
& = & |\mbox{det}A_{\xi}|^{\frac{1}{2}} \displaystyle\int_{\mathbb{R}^{2}} e^{iy_{1}\xi_{1}} f(y) e^{jy_{2}\xi_{2}} \overline{\varphi(-A_{\xi}(y+b))}d\mu_{2}(y)\\
& = & |\mbox{det}A_{-\xi}|^{\frac{1}{2}} \displaystyle\int_{\mathbb{R}^{2}} e^{iy_{1}\xi_{1}} f(y) e^{jy_{2}\xi_{2}} \overline{\varphi(A_{-\xi}(y+b))}d\mu_{2}(y)\\
& = & 	(S_{\varphi}^{Q}f)(-\xi,-b).
\end{eqnarray*}
\item[(vi)] For every $\lambda$ in $\mathbb{R}\setminus\{0\}$, we have
\begin{eqnarray*}
(S_{\varphi}^{Q}f_{\lambda})(\xi,b)
& = & |\mbox{det}A_{\xi}|^{\frac{1}{2}} \displaystyle\int_{\mathbb{R}^{2}} e^{-ix_{1}\xi_{1}} f_{\lambda}(x) e^{-jx_{2}\xi_{2}} \overline{\varphi(A_{\xi}(x-b))}d\mu_{2}(x)\\
& = & |\mbox{det}A_{\xi}|^{\frac{1}{2}} \displaystyle\int_{\mathbb{R}^{2}} e^{-ix_{1}\xi_{1}} f(\lambda x) e^{-jx_{2}\xi_{2}} \overline{\varphi(A_{\xi}(x-b))}d\mu_{2}(x)\\
& = & |\mbox{det}A_{\xi}|^{\frac{1}{2}} \displaystyle\int_{\mathbb{R}^{2}} e^{-iy_{1}\frac{\xi_{1}}{\lambda}} f(y) e^{-jy_{2}\frac{\xi_{2}}{\lambda}} \overline{\varphi(\frac{A_{\xi}}{\lambda}(y-\lambda b))}\dfrac{d\mu_{2}(y)}{|\lambda|^{2}}\\
& = & |\mbox{det}A_{\frac{\xi}{\lambda}}|^{\frac{1}{2}} \displaystyle\int_{\mathbb{R}^{2}} e^{-iy_{1}\frac{\xi_{1}}{\lambda}} f(y) e^{-jy_{2}\frac{\xi_{2}}{\lambda}} \overline{\varphi(A_{\frac{\xi}{\lambda}}(y-\lambda b))}\dfrac{d\mu_{2}(y)}{|\lambda|}\\
& = & \dfrac{(S_{\varphi}^{Q}f)(\frac{\xi}{\lambda},\lambda b)}{|\lambda|}.
\end{eqnarray*}
\end{itemize}
\end{proof}
\begin{theorem} \label{thm1}
Assume that  $\varphi$ satisfies the assumption of theorem \ref{conditions} then for every $f\in L^{2}(\mathbb{R}^{2},\mathbb{H})$, we have
we have
\begin{equation}\label{def}
\mathcal{F}_{Q}((\mathcal{S}_{\varphi}^{Q}f)(\varepsilon,.))(\eta) = \mathcal{F}_{Q}(f)(\eta+\xi) (D_{A_{\xi}^{-1}}\mathcal{F}_{Q}(\check{\overline{\varphi}}))(\eta)
\end{equation}
where $\eta = (\eta_{1},\eta_{2})\in \mathbb{R}^{2}$ and $\xi = (\xi_{1},\xi_{2})\in \mathbb{R}^{2}$.
\end{theorem}
\begin{proof}
From Eq.\ (\ref{stockwell}),\ Eq.\ (\ref{prop1}) and Eq.\ (\ref{prop3})  we have
\begin{equation*}
(S_{\varphi}f)(\xi,b) = M_{\xi}f\divideontimes D_{A_{\xi}}\check{\overline{\varphi}}(b),\ \ \ \mathcal{F}_{Q}(M_{\xi}f)(\eta) = \mathcal{F}_{Q}(f)(\eta+\xi)\ \ \mbox{and}\ \ \mathcal{F}_{Q}(D_{A_{\xi}}\check{\overline{\varphi}})(\eta) = (D_{A_{\xi}^{-1}}\mathcal{F}_{Q}(\check{\overline{\varphi}}))(\eta),
\end{equation*}
and the quaternion window function $\varphi$ in $L^{1}(\mathbb{R}^{2},\mathbb{H})\cap L^{2}(\mathbb{R}^{2},\mathbb{H})$ satisfy the assumption (\ref{hypothese}), then we get the desired result (\ref{def}).
\end{proof}
\begin{theorem} \label{theoremplancherel}
Assume that both admissible quaternion windows $\varphi$ and $\psi$ in $L^{1}(\mathbb{R}^{2},\mathbb{H})\cap L^{2}(\mathbb{R}^{2},\mathbb{H})$ satisfy the assumption of theorem \ref{conditions} , we get
\begin{enumerate}
\item (Parseval's theorem for $\mathcal{S}_{\varphi}^{Q}f$) For all $f$ and $g$  in $L^{2}(\mathbb{R}^{2},\mathbb{H})$,
\begin{equation}\label{parsevalformula}
( \mathcal{S}_{\varphi}^{Q}(f) ,\mathcal{S}_{\psi}^{Q}(g))_{2,\mathbb{R}^{4}} = ( fC_{\{\varphi,\psi\}},g)_{2,\mathbb{R}^{2}}.
\end{equation}
\item (Plancherel's theorem for $\mathcal{S}_{\varphi}^{Q}f$) For every $f$ in $L^{2}(\mathbb{R}^{2},\mathbb{H})$,
\begin{equation}\label{plancherelformula}
\displaystyle\int_{\mathbb{R}^{2}}\int_{\mathbb{R}^{2}} |\mathcal{S}_{\varphi}^{Q}f(\xi,b)|^{2}\ d\mu_{2}(\xi) d\mu_{2}(b) = C_{\varphi} \|f\|_{2,\mathbb{R}^{2}}^{2}.
\end{equation}
\end{enumerate}
\end{theorem}
To give a proof of Theorem \ref{theoremplancherel}, we need the following lemma.
\begin{lemma}\label{lemma1} For all $\varphi$ and $\psi$ in $L^{1}(\mathbb{R}^{2},\mathbb{H})\cap L^{2}(\mathbb{R}^{2},\mathbb{H})$,\ \\ \ \\
$\displaystyle\int_{\mathbb{R}^{2}} \mathcal{F}_{Q}(\check{\overline{\varphi}})(A_{\xi}^{-1}(\zeta-\xi)) \overline{\mathcal{F}_{Q}(\check{\overline{\psi}})(A_{\xi}^{-1}(\zeta-\xi))} \ \dfrac{d\mu_{2}(\xi)}{|\mbox{det}A_{\xi}|}$
\begin{equation}
= \displaystyle\int_{\mathbb{R}^{2}} \mathcal{F}_{Q}(\overline{\varphi})(1-\gamma) \overline{\mathcal{F}_{Q}(\overline{\psi})(1-\gamma)}\ \dfrac{d\mu_{2}(\gamma)}{|\mbox{det}A_{\gamma}|}.
\end{equation}
\end{lemma}
\begin{proof} Let $\zeta$ in $\mathbb{R}^{2}$. Then
\ \\ \ \\
$\displaystyle\int_{\mathbb{R}^{2}} \mathcal{F}_{Q}(\check{\overline{\varphi}})(A_{\xi}^{-1}(\zeta-\xi)) \overline{\mathcal{F}_{Q}(\check{\overline{\psi}})(A_{\xi}^{-1}(\zeta-\xi))} \ \dfrac{d\mu_{2}(\xi)}{|\mbox{det}A_{\xi}|}$
\begin{eqnarray*}
& = & \displaystyle\int_{\mathbb{R}^{2}} \mathcal{F}_{Q}(\check{\overline{\varphi}})(\frac{\zeta_{1} - \xi_{1}}{\xi_{1}},\frac{\zeta_{2} - \xi_{2}}{\xi_{2}}) \overline{\mathcal{F}_{Q}(\check{\overline{\psi}})(\frac{\zeta_{1} - \xi_{1}}{\xi_{1}},\frac{\zeta_{2} - \xi_{2}}{\xi_{2}})}\ \dfrac{d\mu_{2}(\xi)}{|\xi_{1}||\xi_{2}|}\\
& = & \displaystyle\int_{\mathbb{R}^{2}} \mathcal{F}_{Q}(\check{\overline{\varphi}})(\frac{\zeta_{1}}{\xi_{1}}-1,\frac{\zeta_{2}}{\xi_{2}}-1) \overline{\mathcal{F}_{Q}(\check{\overline{\psi}})(\frac{\zeta_{1}}{\xi_{1}}-1,\frac{\zeta_{2} }{\xi_{2}}-1)}\ \dfrac{d\mu_{2}(\xi)}{|\xi_{1}||\xi_{2}|}.
\end{eqnarray*}
Substituting $\gamma = (\gamma_{1},\gamma_{2}) = (\frac{\zeta_{1}}{\xi_{1}},\frac{\zeta_{2} }{\xi_{2}})$ and $\dfrac{d\mu_{2}(\gamma)}{|\mbox{det}A_{\gamma}|} = \dfrac{d\mu_{2}(\xi)}{|\mbox{det}A_{\xi}|} $, we obtain
\ \\ \ \\
$\displaystyle\int_{\mathbb{R}^{2}} \mathcal{F}_{Q}(\check{\overline{\varphi}})(A_{\xi}^{-1}(\zeta-\xi)) \overline{\mathcal{F}_{Q}(\check{\overline{\psi}})(A_{\xi}^{-1}(\zeta-\xi))} \ \dfrac{d\mu_{2}(\xi)}{|\mbox{det}A_{\xi}|}$
\begin{eqnarray*}
& = & \displaystyle\int_{\mathbb{R}^{2}} \mathcal{F}_{Q}(\check{\overline{\varphi}})(\gamma-1) \overline{\mathcal{F}_{Q}(\check{\overline{\psi}})(\gamma-1)}\ \dfrac{d\mu_{2}(\gamma)}{|\mbox{det}A_{\gamma}|}\\
& = & \displaystyle\int_{\mathbb{R}^{2}} \mathcal{F}_{Q}(\overline{\varphi})(1-\gamma) \overline{\mathcal{F}_{Q}(\overline{\psi})(1-\gamma)}\ \dfrac{d\mu_{2}(\gamma)}{|\mbox{det}A_{\gamma}|}\\
& = & C_{\{\varphi,\psi\}}.
\end{eqnarray*}
\end{proof}
\ \\ \\ \
\begin{proof}\textbf{of theorem \ref{theoremplancherel}.} Using Theorem \ref{thm1}, lemma \ref{lemma1} and Fubini's theorem, we get \ \\ \ \\
$Sc\bigg(\displaystyle\int_{\mathbb{R}^{2}}\int_{\mathbb{R}^{2}} (S_{\varphi}^{Q}f)(\xi,b)\overline{(S_{\psi}g)(\xi,b)}\ d\mu_{2}(b)\ d\mu_{2}(\xi) \bigg)$
\begin{eqnarray*}
& = & \displaystyle\int_{\mathbb{R}^{2}} Sc\bigg(\int_{\mathbb{R}^{2}} (S_{\varphi}^{Q}f)(\xi,b)\overline{(S_{\psi}g)(\xi,b)}\ d\mu_{2}(b)\bigg)\ d\mu_{2}(\xi) \\
& = & \displaystyle\int_{\mathbb{R}^{2}} Sc\bigg(\int_{\mathbb{R}^{2}} \mathcal{F}_{Q}(S_{\varphi}^{Q}f)(\xi,.))(\eta)\overline{\mathcal{F}_{Q}(S_{\psi}g)(\xi,.))(\eta)}\ d\mu_{2}(\eta)\bigg)\ d\mu_{2}(\xi) \\
& = & \displaystyle\int_{\mathbb{R}^{2}} Sc\bigg(\int_{\mathbb{R}^{2}} \mathcal{F}_{Q}(f)(\xi+\eta)(D_{A_{\xi}^{-1}}\mathcal{F}_{Q}(\check{\overline{\varphi}}))(\eta)\overline{(D_{A_{\xi}^{-1}}\mathcal{F}_{Q}(\check{\overline{\psi}}))(\eta)} \overline{\mathcal{F}_{Q}(g)(\xi+\eta)}\ d\mu_{2}(\eta)\bigg) \ d\mu_{2}(\xi)\\
& = & |\mbox{det}A_{\xi}^{-1}|\displaystyle\int_{\mathbb{R}^{2}} Sc\bigg(\int_{\mathbb{R}^{2}} \mathcal{F}_{Q}(f)(\zeta)\mathcal{F}_{Q}(\check{\overline{\varphi}})(A_{\xi}^{-1}(\zeta-\xi))\overline{\mathcal{F}_{Q}(\check{\overline{\psi}})(A_{\xi}^{-1}(\zeta-\xi))} \overline{\mathcal{F}_{Q}(g)(\zeta)}\ d\mu_{2}(\zeta)\bigg) \ d\mu_{2}(\xi)\\
& = & Sc\bigg(\displaystyle\int_{\mathbb{R}^{2}}  \mathcal{F}_{Q}(f)(\zeta)\bigg(|\mbox{det}A_{\xi}^{-1}|\int_{\mathbb{R}^{2}}\mathcal{F}_{Q}(\check{\overline{\varphi}})(A_{\xi}^{-1}(\zeta-\xi))\overline{\mathcal{F}_{Q}(\check{\overline{\psi}})(A_{\xi}^{-1}(\zeta-\xi))}\ d\mu_{2}(\xi) \bigg) \overline{\mathcal{F}_{Q}(g)(\zeta)}\ d\mu_{2}(\zeta)\bigg) \\
& = & (\mathcal{F}_{Q}(f) C_{\{\varphi,\psi\}} \mathcal{F}_{Q}(g) )_{2,\mathbb{R}^{2}} \\
& = & ( f C_{\{\varphi,\psi\}}, g )_{2,\mathbb{R}^{2}}.
\end{eqnarray*}
In particular, if $\varphi = \psi$ we have
\begin{equation*}
( \mathcal{S}_{\varphi}^{Q}(f) ,\mathcal{S}_{\psi}^{Q}(g))_{2,\mathbb{R}^{4}} = C_{\varphi}( f,g)_{2,\mathbb{R}^{2}}.
\end{equation*}
And if $f = g$ and $\varphi = \psi$ we get
\begin{equation*}
\|S_{\varphi}^{Q}f\|_{2,\mathbb{R}^{4}} = \sqrt{C_{\varphi}} \|f\|_{2,\mathbb{R}^{2}}.
\end{equation*}
\end{proof}
\begin{theorem}\label{lieb} (Lieb inequality)\\
Let $\varphi$ and $\psi$ be two quaternion window functions satisfy the assumption of theorem \ref{conditions} . For $p\in [1,+\infty[$ and for every $f$, $g$ in $L^{2}(\mathbb{R}^{2},\mathbb{H})$, the function
\begin{equation*}
(\xi,b) \longmapsto (S_{\varphi}^{Q}f)(\xi,b)(S_{\psi}^{Q}g)(\xi,b)
\end{equation*}
belong to $L^{p}(\mathbb{R}^{2}\times\mathbb{R}^{2},\mathbb{H})$ and
\begin{equation*}
\|(S_{\varphi}^{Q}f)\ (S_{\psi}^{Q}g)\|_{p,\mathbb{R}^{4}} \leq (\sqrt{C_{\varphi}C_{\psi}})^{\frac{1}{p}} \|f\|_{2,\mathbb{R}^{2}} \|g\|_{2,\mathbb{R}^{2}} (\|\varphi\|_{2,\mathbb{R}^{2}} \|\psi\|_{2,\mathbb{R}^{2}})^{1-\frac{1}{p}} .
\end{equation*}
\end{theorem}
\begin{proof}
\begin{enumerate}
\item
According to Cauchy-Schwartz inequality and Plancherel's Theorem for the continuous quaternion Stockwell transform $S_{\varphi}^{Q}$ and $S_{\psi}^{Q}$, for every $f,g$ in $L^{2}(\mathbb{R}^{2},\mathbb{H})$,\ \\ \ \\
$\displaystyle\int_{\mathbb{R}^{2}}\int_{\mathbb{R}^{2}} |(S_{\varphi}^{Q}f)(\xi,b) (S_{\psi}^{Q}g)(\xi,b)|\ d\mu_{4}(\xi,b)$
\begin{eqnarray*}
& \leq & \|S_{\varphi}^{Q}f\|_{2,\mathbb{R}^{4}}\ \|S_{\psi}^{Q}g\|_{2,\mathbb{R}^{4}}\\
& = & \sqrt{C_{\varphi}C_{\psi}} \|f\|_{2,\mathbb{R}^{2}}\ \|g\|_{2,\mathbb{R}^{2}},
\end{eqnarray*}
which implies that $S_{\varphi}^{Q}f\ S_{\psi}^{Q}g$ belongs to $L^{1}(\mathbb{R}^{2}\times\mathbb{R}^{2}, \mathbb{H})$ and
\begin{equation}\label{L1}
\|S_{\varphi}^{Q}f\ S_{\psi}^{Q}g\|_{1,\mathbb{R}^{4}} \leq \sqrt{C_{\varphi}C_{\psi}} \|f\|_{2,\mathbb{R}^{2}} \|g\|_{2,\mathbb{R}^{2}}.
\end{equation}
\item
For every $(\xi,b)$ in $\mathbb{R}^{2}\times\mathbb{R}^{2}$, we have
\begin{eqnarray*}
|(S_{\varphi}f)(\xi,b)|
& = & |\mbox{det}A_{\xi}|^{\frac{1}{2}} \bigg|\displaystyle\int_{\mathbb{R}^{2}} e^{-ix_{1}\xi_{1}} f(x)e^{-jx_{2}\xi_{2}} \overline{\varphi(A_{\xi}(x-b))} d\mu_{2}(x)\bigg|\\
& \leq & |\mbox{det}A_{\xi}|^{\frac{1}{2}} \displaystyle\int_{\mathbb{R}^{2}} | f(x)| |\varphi(A_{\xi}(x-b))| d\mu_{2}(x)\\
& \leq & \|f\|_{2,\mathbb{R}^{2}} |\mbox{det}A_{\xi}|^{\frac{1}{2}} \bigg(\displaystyle\int_{\mathbb{R}^{2}} |\varphi(A_{\xi}(x-b))|^{2} d\mu_{2}(x)\bigg)^{\frac{1}{2}}\\
& = & \|f\|_{2,\mathbb{R}^{2}}  \bigg(\displaystyle\int_{\mathbb{R}^{2}} |\varphi(y)|^{2} d\mu_{2}(y)\bigg)^{\frac{1}{2}}\\
& = & \|f\|_{2,\mathbb{R}^{2}} \|\varphi\|_{2,\mathbb{R}^{2}},
\end{eqnarray*}
then
\begin{equation*}
|(S_{\varphi}^{Q}f)(\xi,b)(S_{\psi}^{Q}g)(\xi,b)| \leq \|f\|_{2,\mathbb{R}^{2}} \|\varphi\|_{2,\mathbb{R}^{2}} \|g\|_{2,\mathbb{R}^{2}} \|\psi\|_{2,\mathbb{R}^{2}},
\end{equation*}
which implies that $S_{\varphi}^{Q}f\ S_{\psi}^{Q}g$ belongs to $L^{\infty}(\mathbb{R}^{2}\times\mathbb{R}^{2},\mathbb{H})$ and
\begin{equation}\label{LINFTY}
\|S_{\varphi}^{Q}f\  S_{\psi}^{Q}g\|_{\infty,\mathbb{R}^{4}} \leq \|f\|_{2,\mathbb{R}^{2}} \|\varphi\|_{2,\mathbb{R}^{2}} \|g\|_{2,\mathbb{R}^{2}} \|\psi\|_{2,\mathbb{R}^{2}}.
\end{equation}
\item
By combining relations (\ref{L1}) and (\ref{LINFTY}) and for every $p \in [1,+\infty[$, we get
\ \\ \ \\
$\bigg(\displaystyle\int_{\mathbb{R}^{2}}\int_{\mathbb{R}^{2}} |(S_{\varphi}^{Q}f)(\xi,b)(S_{\psi}^{Q}g)(\xi,b)|^{p} d\mu_{4}(\xi,b)\bigg)^{\frac{1}{p}}$
\begin{eqnarray*}
& = & \bigg(\displaystyle\int_{\mathbb{R}^{2}}\int_{\mathbb{R}^{2}} |(S_{\varphi}^{Q}f)(\xi,b)(S_{\psi}^{Q}g)(\xi,b)|^{p-1+1} d\mu_{4}(\xi,b)\bigg)^{\frac{1}{p}}\\
& \leq & \|S_{\varphi}^{Q}f\  S_{\psi}^{Q}g\|_{\infty,\mathbb{R}^{4}}^{\frac{p-1}{p}} \|S_{\varphi}^{Q}f\ S_{\psi}^{Q}g\|_{1,\mathbb{R}^{4}}^{\frac{1}{p}}\\
& \leq & (\|f\|_{2,\mathbb{R}^{2}} \|\varphi\|_{2,\mathbb{R}^{2}} \|g\|_{2,\mathbb{R}^{2}} \|\psi\|_{2,\mathbb{R}^{2}})^{\frac{p-1}{p}} (\sqrt{C_{\varphi}C_{\psi}} \|f\|_{2,\mathbb{R}^{2}} \|g\|_{2,\mathbb{R}^{2}})^{\frac{1}{p}}\\
& = & (\sqrt{C_{\varphi}C_{\psi}})^{\frac{1}{p}} \|f\|_{2,\mathbb{R}^{2}} \|g\|_{2,\mathbb{R}^{2}} (\|\varphi\|_{2,\mathbb{R}^{2}} \|\psi\|_{2,\mathbb{R}^{2}})^{\frac{p-1}{p}} .
\end{eqnarray*}
\end{enumerate}
\end{proof}
\begin{lemma}
Let $\varphi$ be an admissible  quaternion window function satisfies the assumption of theorem \ref{conditions} . For every $f$ in $L^{2}(\mathbb{R}^{2},\mathbb{H})$, the function $S_{\varphi}^{Q}(f)$ belong to $L^{p}(\mathbb{R}^{2}\times\mathbb{R}^{2},\mathbb{H})$, $p\in [2,+\infty[$ and
\begin{equation*}
\|S_{\varphi}^{Q}f\|_{p,\mathbb{R}^{4}} \leq C_{\varphi}^{\frac{1}{p}} \|\varphi\|_{2,\mathbb{R}^{2}}^{1-\frac{2}{p}} \|f\|_{2,\mathbb{R}^{2}}.
\end{equation*}
\end{lemma}
\begin{proof}
\begin{enumerate}
\item
For $p = +\infty$, we have
\begin{equation*}
|(S_{\varphi}^{Q}f)(\xi,b)| \leq \|f\|_{2,\mathbb{R}^{2}} \|\varphi\|_{2,\mathbb{R}^{2}}
\end{equation*}
and consequently $S_{\varphi}^{Q}f$ belong to $L^{\infty}(\mathbb{R}^{2}\times\mathbb{R}^{2},\mathbb{H})$ and $\|S_{\varphi}^{Q}f\|_{\infty,\mathbb{R}^{4}} \leq \|f\|_{2,\mathbb{R}^{2}} \|\varphi\|_{2,\mathbb{R}^{2}}$.
\item
For every $p\in [1,+\infty[$, the theorem \ref{lieb} implies that for every $f=g$ and $\varphi = \psi$,
\begin{eqnarray*}
\bigg(\displaystyle\int_{\mathbb{R}^{2}} \int_{\mathbb{R}^{2}} |(S_{\varphi}^{Q}f)(\xi,b)|^{2p} d\mu_{4}(\xi,b)\bigg)^{\frac{1}{p}} = C_{\varphi}^{\frac{1}{p}} \|f\|^{2}_{2,\mathbb{R}^{2}}
\|\varphi\|_{2,\mathbb{R}^{2}}^{2-\frac{2}{p}}
\end{eqnarray*}
by the change of variable $q = 2p \in [2, +\infty[$, we get
\begin{eqnarray*}
\bigg(\displaystyle\int_{\mathbb{R}^{2}} \int_{\mathbb{R}^{2}} |(S_{\varphi}^{Q}f)(\xi,b)|^{q} d\mu_{4}(\xi,b)\bigg)^{\frac{2}{q}} = C_{\varphi}^{\frac{2}{q}} \|f\|^{2}_{2,\mathbb{R}^{2}}
\|\varphi\|_{2,\mathbb{R}^{2}}^{2-\frac{4}{q}}
\end{eqnarray*}
finally, we deduce that
\begin{eqnarray*}
\bigg(\displaystyle\int_{\mathbb{R}^{2}} \int_{\mathbb{R}^{2}} |(S_{\varphi}^{Q}f)(\xi,b)|^{q} d\mu_{4}(\xi,b)\bigg)^{\frac{1}{q}} = C_{\varphi}^{\frac{1}{q}} \|f\|_{2,\mathbb{R}^{2}}
\|\varphi\|_{2,\mathbb{R}^{2}}^{1-\frac{2}{q}}.
\end{eqnarray*}
\end{enumerate}
\end{proof}
\section{Uncertainty Principle for $S_{\varphi}^{Q}$}
\begin{definition} (Entropy)\\
According to Shannon, the entropy of a probability density function $P$ on $\mathbb{R}^{2}\times\mathbb{R}^{2}$ is defined by
\begin{equation*}
E(P) = - \displaystyle\int_{\mathbb{R}^{2}}\int_{\mathbb{R}^{2}} ln(P(\xi,b))P(\xi,b)\ d\mu_{2}(\xi) d\mu_{2}(b).
\end{equation*}
\end{definition}
\begin{theorem} (The Beckner's uncertainty principle in terms of entropy for $S_{\varphi}^{Q}$)\\
Let $\varphi$ be a non zero admissible quaternion window function satisfies the assumption of theorem \ref{conditions} . Then  for all $f$ in $L^{2}(\mathbb{R}^{2},\mathbb{H})$ with $f\neq 0$; we have
\begin{equation}\label{entropyuncertainty}
E(|(S_{\varphi}^{Q}f)|^{2}) \geq \|f\|_{2,\mathbb{R}^{2}}^{2} \|\varphi\|_{2,\mathbb{R}^{2}}^{2}\ ln\bigg(\dfrac{1}{\|f\|_{2,\mathbb{R}^{2}}^{2} \|\varphi\|_{2,\mathbb{R}^{2}}^{2}}\bigg).
\end{equation}
\end{theorem}
\begin{proof} Assume that $\|f\|_{2,\mathbb{R}^{2}} = \|\varphi\|_{2,\mathbb{R}^{2}} = 1$, then by relation (\ref{LINFTY}) we deduce that
\begin{equation*}
\forall(\xi,b) \in \mathbb{R}^{2}\times\mathbb{R}^{2},\ |(S_{\varphi}^{Q}f)(\xi,b)| \leq \|(S_{\varphi}^{Q}f)\|_{\infty,\mathbb{R}^{4}} \leq \|f\|_{2,\mathbb{R}^{2}} \|\varphi\|_{2,\mathbb{R}^{2}} = 1
\end{equation*}
then $ln(|S_{\varphi}^{Q}f(\xi,b)|) \leq 0$ in particular $E(|(S_{\varphi}^{Q}f)|) \geq 0$.
\begin{itemize}
\item
Therefore if the entropy $E(|(S_{\varphi}^{Q}f)|) = +\infty$ then the inequality (\ref{entropyuncertainty}) holds trivially.
\item
Suppose now that the entropy $E(|(S_{\varphi}^{Q}f)|) < +\infty$ and let $0<x<1$ and $M_{x}$ be the function
\end{itemize}
 defined on $]2,3]$ by $M_{x}(p) = \dfrac{x^{p} - x^{2}}{p-2}$, then
 \begin{equation*}
 \forall p \in ]2,3],\ \dfrac{dM_{x}}{dp}(p) = \dfrac{(p-2)x^{p}ln(x) - x^{p} + x^{2}}{(p-2)^{2}}.
 \end{equation*}
 The sign of $\dfrac{dM_{x}}{dp}(p)$ is the same as that of the function
 \begin{equation*}
 N_{x}(p) = (p-2)x^{p}ln(x) - x^{p} + x^{2}.
  \end{equation*}
  For every $0<x<1$, the function $N_{x}$ is differentiable on $\mathbb{R}$, especially on $]2,3]$, and its derivative is
  \begin{equation*}
  \dfrac{dN_{x}}{dp}(p) = (p-2) (ln(x))^{2}x^{p}.
  \end{equation*}
  We have that, for all $0<x<1$, $\dfrac{dN_{x}}{dp}(p)$ is positive on $]2,3]$, then $N_{x}$ is increasing on $]2,3]$.\  \\ \ \\
  For all $0< x< 1$, $\displaystyle\lim_{p\longmapsto 2^{+}} N_{x}(p) = N_{x}(2) = 0$, then $N_{x}$ is positive which implies that $\dfrac{dM_{x}}{dp}(p)$ is positive also on $]2,3]$ and consequently $p \longmapsto M_{x}(p)$ is increasing on $]2,3]$.\\
 In particular,
 \begin{equation*}
 \forall p \in ]2,3],\ x^{2} \ ln(x) = \displaystyle\lim_{p\longmapsto 2^{+}}  \dfrac{x^{p} - x^{2}}{p-2} \leq M_{x}(p)
 \end{equation*}
 hence
 \begin{equation*}
  \forall p \in ]2,3],\ 0 \leq \dfrac{x^{2} - x^{p}}{p-2} \leq -x^{2} \ ln(x).
 \end{equation*}
 We have already observed that for all every $(\xi,b)\in \mathbb{R}^{2}\times \mathbb{R}^{2}$, $0 \leq |(S_{\varphi}^{Q}f)(\xi,b)| \leq 1$; then we get for every $p \in ]2,3]$
 \begin{equation}\label{sphi}
 0 \leq \dfrac{|(S_{\varphi}^{Q}f)(\xi,b)|^{2} - |(S_{\varphi}^{Q}f)(\xi,b)|^{p}}{p - 2} \leq - |(S_{\varphi}^{Q}f)(\xi,b)|^{2} \ ln(|(S_{\varphi}^{Q}f)(\xi,b)|).
 \end{equation}
 Let F be the function defined on $[2,+\infty[$ by
 \begin{equation*}
 \mbox{F}(p) = \bigg(\displaystyle\int_{\mathbb{R}^{2}}\int_{\mathbb{R}^{2}} |(S_{\varphi}^{Q}f)(\xi,b)|^{p} d\mu_{4}(\xi,b)\bigg) - C_{\varphi}.
 \end{equation*}
 According to Lieb inequality, we know that for every $2 \leq p < +\infty$, the continuous quaternion Stockwell transform $(S_{\varphi}^{Q}f)$ belongs to $L^{p}(\mathbb{R}^{2}\times\mathbb{R}^{2},\mathbb{H})$ and we have
 \begin{equation}\label{sphi1}
 \|(S_{\varphi}^{Q}f)\|_{p,\mathbb{R}^{4}} \leq C_{\varphi}^{\frac{1}{p}} \|f\|_{2,\mathbb{R}^{2}} \|\varphi\|_{2,\mathbb{R}^{2}}^{1-\frac{2}{p}} = C_{\varphi}^{\frac{1}{p}}.
 \end{equation}
 Then relation (\ref{sphi1}) implies that $\mbox{F}(p) \leq 0$ for every $p \in [2, +\infty[$ and by Plancherel's formula, we have
 \begin{equation*}
 \mbox{F}(2) =\bigg(\displaystyle\int_{\mathbb{R}^{2}}\int_{\mathbb{R}^{2}} |(S_{\varphi}^{Q}f)(\xi,b)|^{2} d\mu_{4}(\xi,b)\bigg) - C_{\varphi}  = \|(S_{\varphi}^{Q}f)\|_{2,\mathbb{R}^{4}} - C_{\varphi} = 0.
\end{equation*}
Therefore $\bigg(\dfrac{d\mbox{F}}{dp}\bigg)_{p=2^{+}} \leq 0$ whenever this derivative is well defined.\\
On the other hand, we have for every $p \in ]2,3]$ and for $(\xi,b) \in \mathbb{R}^{2}\times\mathbb{R}^{2}$
\begin{equation*}
\bigg|\dfrac{|(S_{\varphi}^{Q}f)(\xi,b)|^{p} - |(S_{\varphi}^{Q}f)(\xi,b)|^{2}|}{p - 2}\bigg|  \leq |(S_{\varphi}^{Q}f)(\xi,b)|^{2} \ ln(|(S_{\varphi}^{Q}f)(\xi,b)|).
\end{equation*}
Then\ \\ \ \\
$\displaystyle\int_{\mathbb{R}^{2}}\int_{\mathbb{R}^{2}} \bigg|\dfrac{||(S_{\varphi}^{Q}f)(\xi,b)|^{p} - |(S_{\varphi}^{Q}f)(\xi,b)|^{2}}{p - 2}\bigg| d\mu_{4}(\xi,b)$
\begin{eqnarray*}
& \leq & - \displaystyle\int_{\mathbb{R}^{2}}\int_{\mathbb{R}^{2}} |(S_{\varphi}^{Q}f)(\xi,b)|^{2} \ ln(|(S_{\varphi}^{Q}f)(\xi,b)|)  d\mu_{4}(\xi,b) \\
& = & \dfrac{1}{2} E(|(S_{\varphi}^{Q}f)|^{2}) < + \infty.
\end{eqnarray*}
Moreover, for every $p\in ]3,+\infty[$ and for every $(\xi,b)\in \mathbb{R}^{2}\times\mathbb{R}^{2}$
\begin{equation*}
\dfrac{||S_{\varphi}^{Q}f)(\xi,b)|^{p} - |S_{\varphi}^{Q}f)(\xi,b)|^{2}|}{p-2} \leq 2|(S_{\varphi}^{Q}f)(\xi,b)|^{2}
\end{equation*}
and consequently
\begin{equation*}
\displaystyle\int_{\mathbb{R}^{2}}\int_{\mathbb{R}^{2}} \bigg|\dfrac{|S_{\varphi}^{Q}f)(\xi,b)|^{p} - |S_{\varphi}^{Q}f)(\xi,b)|^{2}}{p-2} \bigg| d\mu_{4}(\xi,b) \leq 2 \|(S_{\varphi}^{Q}f)\|_{2,\mathbb{R}^{4}} = 2 <+\infty.
\end{equation*}
Using relation (\ref{sphi}) and Lebesgue's dominated convergence theorem, we have\ \\ \ \\
$\bigg(\dfrac{d}{dp} \displaystyle\int_{\mathbb{R}^{2}}\int_{\mathbb{R}^{2}} |(S_{\varphi}^{Q}f)(\xi,b)|^{p}\ d\mu_{4}(\xi,b) \bigg)_{p=2^{+}}$
\begin{eqnarray*}
& = & \displaystyle\lim_{p\longmapsto 2^{+}} \displaystyle\int_{\mathbb{R}^{2}}\int_{\mathbb{R}^{2}} \dfrac{|S_{\varphi}^{Q}f)(\xi,b)|^{p} - |S_{\varphi}^{Q}f)(\xi,b)|^{2}}{p-2}\ d\mu_{4}(\xi,b)\\
& = &  \displaystyle\int_{\mathbb{R}^{2}}\int_{\mathbb{R}^{2}} \displaystyle\lim_{p\longmapsto 2^{+}} \dfrac{|S_{\varphi}^{Q}f)(\xi,b)|^{p} - |S_{\varphi}^{Q}f)(\xi,b)|^{2}}{p-2}\ d\mu_{4}(\xi,b)\\
& = & \dfrac{1}{2} \displaystyle\int_{\mathbb{R}^{2}}\int_{\mathbb{R}^{2}} ln(|S_{\varphi}^{Q}f)(\xi,b)|^{2})\ |S_{\varphi}^{Q}f)(\xi,b)|^{2}\ d\mu_{4}(\xi,b)\\
& = & -\dfrac{1}{2} E(|(S_{\varphi}^{Q}f)|^{2}),
\end{eqnarray*}
and consequently
\begin{equation*}
\bigg(\dfrac{d F}{dp}\bigg)_{p=2^{+}} = -\dfrac{1}{2} E(|(S_{\varphi}^{Q}f)|^{2}) - \bigg(\dfrac{dC_{\varphi}}{dp}\bigg)_{p=2^{+}} = -\dfrac{1}{2} E(|(S_{\varphi}^{Q}f)|^{2}).
\end{equation*}
Which gives
\begin{equation*}
E(|(S_{\varphi}^{Q}f)|^{2}) \geq 0.
\end{equation*}
So (\ref{entropyuncertainty}) is true for $\|f\|_{2,\mathbb{R}^{2}} = \|\varphi\|_{2,\mathbb{R}^{2}} = 1$. For generic $f, \varphi \neq 0$, let  $g = \dfrac{f}{\|f\|_{2,\mathbb{R}^{2}}}$ and $\psi = \dfrac{\varphi}{\|\varphi\|_{2,\mathbb{R}^{2}}}$, so that $\|g\|_{2,\mathbb{R}^{2}} = \|\psi\|_{2,\mathbb{R}^{2}} = 1$ and $E(|(S_{\psi}^{Q}g)|^{2}) \geq 0$.\ \\ \ \\
Since
\begin{equation*}
(S_{\psi}^{Q}g) = \dfrac{(S_{\varphi}^{Q}f)}{\|f\|_{2,\mathbb{R}^{2}} \|\varphi\|_{2,\mathbb{R}^{2}}}
\end{equation*}
by using Plancherel's formula (\ref{plancherelformula}) and Fubini's theorem, we get
\begin{eqnarray*}
E(|(S_{\psi}^{Q}g)|^{2})
& = & -  \displaystyle\int_{\mathbb{R}^{2}}\int_{\mathbb{R}^{2}} ln(|(S_{\psi}^{Q}g)(\xi,b)|^{2}) \ |(S_{\psi}^{Q}g)(\xi,b)|^{2} d\mu_{4}(\xi,b)\\
& = & -  \displaystyle\int_{\mathbb{R}^{2}}\int_{\mathbb{R}^{2}}
\big(ln(|(S_{\varphi}^{Q}f)(\xi,b)|^{2}) - ln(\|f\|_{2,\mathbb{R}^{2}}^{2} \|\varphi\|_{2,\mathbb{R}^{2}}^{2})\big) \dfrac{|(S_{\varphi}^{Q}f)(\xi,b)|^{2}}{\|f\|_{2,\mathbb{R}^{2}}^{2} \|\varphi\|_{2,\mathbb{R}^{2}}^{2}} d\mu_{4}(\xi,b)\\
& = & \dfrac{E(|(S_{\varphi}^{Q}f)|^{2})}{\|f\|_{2,\mathbb{R}^{2}}^{2} \|\varphi\|_{2,\mathbb{R}^{2}}^{2}} + ln(\|f\|_{2,\mathbb{R}^{2}}^{2} \|\varphi\|_{2,\mathbb{R}^{2}}^{2})\\
& = & \dfrac{E(|(S_{\varphi}^{Q}f)|^{2})}{\|f\|_{2,\mathbb{R}^{2}}^{2} \|\varphi\|_{2,\mathbb{R}^{2}}^{2}} - ln\bigg(\dfrac{1}{\|f\|_{2,\mathbb{R}^{2}}^{2} \|\varphi\|_{2,\mathbb{R}^{2}}^{2}}\bigg)\\
& \geq & 0;
\end{eqnarray*}
we deduce that
\begin{equation*}
E(|(S_{\varphi}^{Q}f)|^{2}) \geq \|f\|_{2,\mathbb{R}^{2}}^{2} \|\varphi\|_{2,\mathbb{R}^{2}}^{2}\ ln\bigg(\dfrac{1}{\|f\|_{2,\mathbb{R}^{2}}^{2} \|\varphi\|_{2,\mathbb{R}^{2}}^{2}}\bigg).
\end{equation*}
\end{proof}

\begin{theorem} (The generalized Heisenberg uncertainty principle for $S_{\varphi}^{Q}$)\\
Let $p$ and $q$ be two positive real numbers. Then there exists a non-negative constant $D_{p,q}$ such that for every non zero quaternion window function $\varphi$ satisfies the assumption of theorem \ref{conditions}  and for every non zero quaternion function $f$ in $L^{2}(\mathbb{R}^{2},\mathbb{H})$ we have
\begin{eqnarray*}
\bigg(\displaystyle\int_{\mathbb{R}^{2}}\int_{\mathbb{R}^{2}} |\xi|^{p} |(S_{\varphi}^{Q}f)(\xi,b)|^{2} d\mu_{4}(\xi,b)\bigg)^{\frac{q}{p+q}} \bigg(\displaystyle\int_{\mathbb{R}^{2}}\int_{\mathbb{R}^{2}} |b|^{q} |(S_{\varphi}^{Q}f)(\xi,b)|^{2} d\mu_{4}(\xi,b)\bigg)^{\frac{p}{p+q}}\\
\geq D_{p,q} \|f\|_{2,\mathbb{R}^{2}}^{2}\|\varphi\|_{2,\mathbb{R}^{2}}^{2};
\end{eqnarray*}
where
$D_{p,q}  = \dfrac{2}{ep} \bigg(\dfrac{p}{q}\bigg)^{\frac{p}{p+q}}
\bigg(\dfrac{pq}{\Gamma(\frac{2}{p})\Gamma(\frac{2}{q})}\bigg)^{\frac{pq}{2(p+q)}}.$
\end{theorem}
\begin{proof}
\ \\
Assume that $\|f\|_{2,\mathbb{R}^{2}} = \|\varphi\|_{2,\mathbb{R}^{2}} = 1$ and let $A_{t,p,q}$ be the function defined on $\mathbb{R}^{2}\times\mathbb{R}^{2}$ by
\begin{equation*}
A_{t,p,q}(\xi,b) = \dfrac{B_{p,q} e^{-\frac{|\xi|^{p}+|b|^{q}}{t}}}{t^{\frac{2}{p}+\frac{2}{q}}} \ \ \ \mbox{and}\ \ \ B_{p,q} = \dfrac{pq}{\Gamma(\frac{2}{p})\Gamma(\frac{2}{q})},
\end{equation*}
where $\Gamma(.)$ the Gamma function.
\begin{itemize}
\item[(i)]
We see that
\begin{equation*}
\displaystyle\int_{\mathbb{R}^{2}}\int_{\mathbb{R}^{2}} A_{t,p,q}(\xi,b)\ d\mu_{4}(\xi,b) = 1
\end{equation*}
in particular $d\sigma_{t,p,q}(\xi,b) = A_{t,p,q}(\xi,b)\ d\mu_{4}(\xi,b)$ is a probability measure on $\mathbb{R}^{2}\times\mathbb{R}^{2}$.
\item[(ii)]
The function $\mbox{F}(t) = t\ ln(t)$ is a convex function over $]0,+\infty[$.
\item[(iii)]
$\mbox{G}(\xi,b) = \dfrac{|S_{\varphi}^{Q}f)(\xi,b)|^{2}}{A_{t,p,q}}$ is a real-valued function, integrable with respect to the measure $d\sigma_{t,p,q}(\xi,b)$ on $\mathbb{R}^{2}\times\mathbb{R}^{2}$.
\end{itemize}
Hence according to Jensen Inequality we get
\begin{equation*}
\mbox{F}\bigg(\displaystyle\int_{\mathbb{R}^{2}}\int_{\mathbb{R}^{2}} \mbox{G}(\xi,b) \ d\sigma_{t,p,q}(\xi,b)\bigg) \leq \displaystyle\int_{\mathbb{R}^{2}}\int_{\mathbb{R}^{2}} \mbox{F}(\mbox{G}(\xi,b)) \ d\sigma_{t,p,q}(\xi,b)
\end{equation*}
so
\begin{equation*}
\mbox{F}\bigg(\displaystyle\int_{\mathbb{R}^{2}}\int_{\mathbb{R}^{2}} G(\xi,b) \ d\sigma_{t,p,q}(\xi,b)\bigg) = \mbox{F}\bigg(\displaystyle\int_{\mathbb{R}^{2}}\int_{\mathbb{R}^{2}} \dfrac{|S_{\varphi}^{Q}f)(\xi,b)|^{2}}{A_{t,p,q}} \ d\sigma_{t,p,q}(\xi,b)\bigg) = \mbox{F}(\|S_{\varphi}^{Q}f\|_{2,\mathbb{R}^{2}}^{2}) = \mbox{F}(1) = 0
\end{equation*}
and
\begin{eqnarray*}
\displaystyle\int_{\mathbb{R}^{2}}\int_{\mathbb{R}^{2}} \mbox{F}(\mbox{G}(\xi,b)) \ d\sigma_{t,p,q}(\xi,b) = \displaystyle\int_{\mathbb{R}^{2}}\int_{\mathbb{R}^{2}} \dfrac{|S_{\varphi}^{Q}f)(\xi,b)|^{2}}{A_{t,p,q}} \ ln\bigg(\dfrac{|S_{\varphi}^{Q}f)(\xi,b)|^{2}}{A_{t,p,q}}\bigg)\ d\sigma_{t,p,q}(\xi,b).
\end{eqnarray*}
Then
\begin{equation*}
0 \leq \displaystyle\int_{\mathbb{R}^{2}}\int_{\mathbb{R}^{2}} \dfrac{|S_{\varphi}^{Q}f)(\xi,b)|^{2}}{A_{t,p,q}} \ ln\bigg(\dfrac{|S_{\varphi}^{Q}f)(\xi,b)|^{2}}{A_{t,p,q}}\bigg)\ d\sigma_{t,p,q}(\xi,b).
\end{equation*}
which implies in terms of entropy that for every positive real number $t$,
\begin{eqnarray*}
0 & \leq & \displaystyle\int_{\mathbb{R}^{2}}\int_{\mathbb{R}^{2}} \dfrac{|S_{\varphi}^{Q}f)(\xi,b)|^{2}}{A_{t,p,q}} \ ln\bigg(\dfrac{|S_{\varphi}^{Q}f)(\xi,b)|^{2}}{A_{t,p,q}}\bigg)\ d\sigma_{t,p,q}(\xi,b)\\
& = & \displaystyle\int_{\mathbb{R}^{2}}\int_{\mathbb{R}^{2}} |(S_{\varphi}^{Q}f)(\xi,b)|^{2} \ ln(|S_{\varphi}^{Q}f)(\xi,b)|^{2})\ d\mu_{4}(\xi,b) -  \displaystyle\int_{\mathbb{R}^{2}}\int_{\mathbb{R}^{2}} |S_{\varphi}^{Q}f)(\xi,b)|^{2} \ ln(A_{t,p,q}(\xi,b))\ \ d\mu_{4}(\xi,b)\\
& = & - E(|(S_{\varphi}^{Q}f)|^{2})  -  \displaystyle\int_{\mathbb{R}^{2}}\int_{\mathbb{R}^{2}} |S_{\varphi}^{Q}f)(\xi,b)|^{2} \ ln(A_{t,p,q}(\xi,b))\ \ d\mu_{4}(\xi,b)\\
& = & - E(|(S_{\varphi}^{Q}f)|^{2}) + \bigg(ln(t^{\frac{2}{p}+\frac{2}{q}}) - ln(B_{p,q})\bigg)\|(S_{\varphi}^{Q}f)\|_{2,\mathbb{R}^{4}}^{2} + \dfrac{1}{t}\bigg(\displaystyle\int_{\mathbb{R}^{2}}\int_{\mathbb{R}^{2}} (|\xi|^{p}+|b|^{q})|S_{\varphi}^{Q}f)(\xi,b)|^{2} \ d\mu_{4}(\xi,b)\bigg).\\
& = & - E(|(S_{\varphi}^{Q}f)|^{2}) + \bigg(ln(t^{\frac{2}{p}+\frac{2}{q}}) - ln(B_{p,q})\bigg) + \dfrac{1}{t}\bigg(\displaystyle\int_{\mathbb{R}^{2}}\int_{\mathbb{R}^{2}} (|\xi|^{p}+|b|^{q})|S_{\varphi}^{Q}f)(\xi,b)|^{2} \ d\mu_{4}(\xi,b)\bigg).\\
\end{eqnarray*}
Therefore
\begin{equation*}
E(|(S_{\varphi}^{Q}f)|^{2}) + ln(B_{p,q}) \leq ln(t^{\frac{2}{p}+\frac{2}{q}}) + \dfrac{1}{t}\bigg(\displaystyle\int_{\mathbb{R}^{2}}\int_{\mathbb{R}^{2}} (|\xi|^{p}+|b|^{q})|S_{\varphi}^{Q}f)(\xi,b)|^{2} \ d\mu_{4}(\xi,b)\bigg).
\end{equation*}
By (\ref{entropyuncertainty}), we get
\begin{eqnarray*}
\|f\|_{2,\mathbb{R}^{2}}^{2} \|\varphi\|_{2,\mathbb{R}^{2}}^{2}\ ln\bigg(\dfrac{1}{\|f\|_{2,\mathbb{R}^{2}}^{2} \|\varphi\|_{2,\mathbb{R}^{2}}^{2}}\bigg) + ln(B_{p,q}) \leq ln(t^{\frac{2}{p}+\frac{2}{q}}) + \dfrac{1}{t}\bigg(\displaystyle\int_{\mathbb{R}^{2}}\int_{\mathbb{R}^{2}} (|\xi|^{p}+|b|^{q})|S_{\varphi}^{Q}f)(\xi,b)|^{2} \ d\mu_{4}(\xi,b)\bigg).
\end{eqnarray*}
Knowing that $\|f\|_{2,\mathbb{R}^{2}}^{2} = \|\varphi\|_{2,\mathbb{R}^{2}}^{2} = 1$, we deduce that
\begin{equation*}
t\bigg(ln(B_{p,q}) - ln(t^{\frac{2}{p}+\frac{2}{q}})\bigg) \leq \bigg(\displaystyle\int_{\mathbb{R}^{2}}\int_{\mathbb{R}^{2}} (|\xi|^{p}+|b|^{q})|S_{\varphi}^{Q}f)(\xi,b)|^{2} \ d\mu_{4}(\xi,b)\bigg).
\end{equation*}
However the expression $t\bigg(ln(B_{p,q}) - ln(t^{\frac{2}{p}+\frac{2}{q}})\bigg)$ attains its upper bound if
\begin{equation*}
0 = \dfrac{d}{dt}\bigg(t(ln(B_{p,q}) - ln(t^{\frac{2}{p}+\frac{2}{q}}))\bigg) = ln(B_{p,q}) - ln(t^{\frac{2}{p}+\frac{2}{q}}) - \bigg(\frac{2}{p}+\frac{2}{q}\bigg);
\end{equation*}
at the point $t = \dfrac{B_{p,q}^{\frac{pq}{2(p+q)}}}{e}$, which implies that
\begin{equation*}
\bigg(\displaystyle\int_{\mathbb{R}^{2}}\int_{\mathbb{R}^{2}} (|\xi|^{p}+|b|^{q})|S_{\varphi}^{Q}f)(\xi,b)|^{2} \ d\mu_{4}(\xi,b)\bigg) \geq C_{p,q};
\end{equation*}
where $C_{p,q} = \dfrac{2}{e}\bigg(\dfrac{p+q}{pq}\bigg)B_{p,q}^{\frac{pq}{2(p+q)}}$.\\ \ \\
Suppose now that $f \neq 0$,\ $g = \dfrac{f}{\|f\|_{2,\mathbb{R}^{2}}} $ and $\Psi = \dfrac{\varphi}{\|\varphi\|_{2,\mathbb{R}^{2}}}$, so that $\|g\|_{2,\mathbb{R}^{2}} = \|\Psi\|_{2,\mathbb{R}^{2}} = 1$. By the previous calculations we have
\begin{equation*}
\bigg(\displaystyle\int_{\mathbb{R}^{2}}\int_{\mathbb{R}^{2}} (|\xi|^{p}+|b|^{q})|S_{\Psi}^{Q}g)(\xi,b)|^{2} \ d\mu_{4}(\xi,b)\bigg) \geq C_{p,q}.
\end{equation*}
Using the relation
\begin{equation*}
|S_{\Psi}^{Q}g)(\xi,b)|=\dfrac{|S_{\varphi}^{Q}f)(\xi,b)|}{\|f\|_{2,\mathbb{R}^{2}} \|\varphi\|_{2,\mathbb{R}^{2}}};
\end{equation*}
we get
\begin{equation*}
\bigg(\displaystyle\int_{\mathbb{R}^{2}}\int_{\mathbb{R}^{2}} (|\xi|^{p}+|b|^{q})|(S_{\varphi}^{Q}f)(\xi,b)|^{2} \ d\mu_{4}(\xi,b)\bigg) \geq C_{p,q} \|f\|_{2,\mathbb{R}^{2}}^{2} \|\varphi\|_{2,\mathbb{R}^{2}}^{2}.
\end{equation*}
Now, for every positive real number $\lambda$ the dilates $f_{\lambda}$ belongs to $L^{2}(\mathbb{R}^{2},\mathbb{H})$, we have
\begin{equation*}
\bigg(\displaystyle\int_{\mathbb{R}^{2}}\int_{\mathbb{R}^{2}} (|\xi|^{p}+|b|^{q})|(S_{\varphi}^{Q}f_{\lambda})(\xi,b)|^{2} \ d\mu_{4}(\xi,b)\bigg) \geq C_{p,q} \|f_{\lambda}\|_{2,\mathbb{R}^{2}}^{2} \|\varphi\|_{2,\mathbb{R}^{2}}^{2}.
\end{equation*}
Then by relation (\ref{scaling}), we have\ \\  \ \\
$\displaystyle\int_{\mathbb{R}^{2}}\int_{\mathbb{R}^{2}} (|\xi|^{p}+|b|^{q})|(S_{\varphi}^{Q}f_{\lambda})(\xi,b)|^{2} \ d\mu_{4}(\xi,b)$
\begin{eqnarray*}
& = & \displaystyle\int_{\mathbb{R}^{2}}\int_{\mathbb{R}^{2}} (|\xi|^{p}+|b|^{q})|(S_{\varphi}^{Q}f)(\frac{\xi}{\lambda},\lambda b)|^{2} \ \dfrac{d\mu_{4}(\xi,b)}{|\lambda|^{2}}\\
& = & \displaystyle\int_{\mathbb{R}^{2}}\int_{\mathbb{R}^{2}}  \bigg(|\lambda \xi|^{p} + \bigg|\dfrac{b}{\lambda}\bigg|^{q}\bigg) |(S_{\varphi}^{Q}f)(\xi,b)|^{2}\ \dfrac{d\mu_{4}(\xi,b)}{|\lambda|^{2}};
\end{eqnarray*}
and
\begin{equation*}
\|f_{\lambda}\|_{2,\mathbb{R}^{2}}^{2} = \dfrac{\|f\|_{2,\mathbb{R}^{2}}^{2}}{|\lambda|^{2}}.
\end{equation*}
Therefore for every positive real number $\lambda$
\begin{equation*}
\displaystyle\int_{\mathbb{R}^{2}}\int_{\mathbb{R}^{2}} \bigg(|\lambda \xi|^{p} + \bigg|\frac{b}{\lambda}\bigg|^{q}\bigg) |(S_{\varphi}^{Q}f)(\xi,b)|^{2}\ d\mu_{4}(\xi,b) \geq C_{p,q} \|f\|_{2,\mathbb{R}^{2}}^{2}\|\varphi\|_{2,\mathbb{R}^{2}}^{2};
\end{equation*}
in particular, the inequality holds at the critical point $\lambda$ where
\begin{equation*}
\dfrac{d}{d\lambda}\bigg(\displaystyle\int_{\mathbb{R}^{2}}\int_{\mathbb{R}^{2}} \bigg(|\lambda \xi|^{p} + \bigg|\frac{b}{\lambda}\bigg|^{q}\bigg) |(S_{\varphi}^{Q}f)(\xi,b)|^{2}\ d\mu_{4}(\xi,b)\bigg) = 0;
\end{equation*}
then for
\begin{equation*}
\lambda = \bigg(\dfrac{q}{p}\bigg)^{\frac{1}{p+q}}\bigg(\displaystyle\int_{\mathbb{R}^{2}}\int_{\mathbb{R}^{2}} |\xi|^{p} |(S_{\varphi}^{Q}f)(\xi,b)|^{2} d\mu_{4}(\xi,b)\bigg)^{-\frac{1}{p+q}} \bigg(\displaystyle\int_{\mathbb{R}^{2}}\int_{\mathbb{R}^{2}} |b|^{q} |(S_{\varphi}^{Q}f)(\xi,b)|^{2} d\mu_{4}(\xi,b)\bigg)^{\frac{1}{p+q}}
\end{equation*}
we have that
\begin{eqnarray*}
\bigg[\bigg(\dfrac{q}{p}\bigg)^{\frac{p}{p+q}} + \bigg(\dfrac{p}{q}\bigg)^{\frac{q}{p+q}}\bigg] \bigg(\displaystyle\int_{\mathbb{R}^{2}}\int_{\mathbb{R}^{2}} |\xi|^{p} |(S_{\varphi}^{Q}f)(\xi,b)|^{2} d\mu_{4}(\xi,b)\bigg)^{\frac{q}{p+q}} \bigg(\displaystyle\int_{\mathbb{R}^{2}}\int_{\mathbb{R}^{2}} |b|^{q} |(S_{\varphi}^{Q}f)(\xi,b)|^{2} d\mu_{4}(\xi,b)\bigg)^{\frac{p}{p+q}}\\
\geq C_{p,q} \|f\|_{2,\mathbb{R}^{2}}^{2}\|\varphi\|_{2,\mathbb{R}^{2}}^{2};
\end{eqnarray*}
therefore
\begin{eqnarray*}
\bigg(\displaystyle\int_{\mathbb{R}^{2}}\int_{\mathbb{R}^{2}} |\xi|^{p} |(S_{\varphi}^{Q}f)(\xi,b)|^{2} d\mu_{4}(\xi,b)\bigg)^{\frac{q}{p+q}} \bigg(\displaystyle\int_{\mathbb{R}^{2}}\int_{\mathbb{R}^{2}} |b|^{q} |(S_{\varphi}^{Q}f)(\xi,b)|^{2} d\mu_{4}(\xi,b)\bigg)^{\frac{p}{p+q}}\\
\geq D_{p,q} \|f\|_{2,\mathbb{R}^{2}}^{2}\|\varphi\|_{2,\mathbb{R}^{2}}^{2};
\end{eqnarray*}
where
\begin{eqnarray*}
D_{p,q}
& = & \dfrac{2}{ep} \bigg(\dfrac{p}{q}\bigg)^{\frac{p}{p+q}}
\bigg(\dfrac{pq}{\Gamma(\frac{2}{p})\Gamma(\frac{2}{q})}\bigg)^{\frac{pq}{2(p+q)}}.
\end{eqnarray*}
\end{proof}
\begin{theorem} (Local uncertainty principle for $S_{\varphi}^{Q}$)\\
Let $\alpha$ and $p$ be two positive real numbers such that $\alpha > 0$ and $p\geq 1$, then there is a non-negative constant $M_{\alpha,p}$ such that for every non zero admissible quaternion window function $ \varphi$ satisfies the assumption of theorem \ref{conditions} , for every function $f$ in $L^{2}(\mathbb{R}^{2},\mathbb{H})$ and for every finite measurable subset $\Sigma$ of $\mathbb{R}^{2}\times\mathbb{R}^{2}$, we have
\begin{eqnarray}
\|(S_{\varphi}^{Q}f)\chi_{\Sigma}\|_{p,\mathbb{R}^{d}} \leq M_{\alpha,p}\  (\|f\|_{2,\mathbb{R}^{2}}\|\varphi\|_{2,\mathbb{R}^{2}})^{1-\frac{4(\alpha+1)}{(3\alpha+2)(p+1)}}(\mu_{\mathbb{R}^{4}}(\Sigma))^{\frac{1}{p(p+1)}}  \||(\xi,b)|^{-\alpha}(S_{\varphi}^{Q}f)\|_{2,\mathbb{R}^{2}}^{\frac{4(\alpha+1)}{(3\alpha+2)(p+1)}}
\end{eqnarray}
where
\begin{equation*}
M_{\alpha,p} = \bigg(\dfrac{3\alpha+2}{2\alpha}\bigg)\bigg(\dfrac{2\alpha}{\alpha+2}\bigg)^{\frac{\alpha+2}{3\alpha+2}}
\bigg(\dfrac{1}{2\sqrt{\alpha+2}}\bigg)^{\frac{2\alpha}{(3\alpha+2)(p+1)}}.
\end{equation*}
\end{theorem}
\begin{proof}
In this proof, we write $B_{r} = \mathbb{R}^{2}\times\mathbb{R}^{2}\cap\{(\xi,b):|(\xi,b)|\leq r\}$ and $B_{r}^{c} = \mathbb{R}^{2}\times\mathbb{R}^{2}\setminus B_{r}$.\\
Without loss of generality we can assume that $\|f\|_{2,\mathbb{R}^{2}} = \|g\|_{2,\mathbb{R}^{2}} = 1$, then for every positive real number $r$, we have
\begin{equation*}
\|(S_{\varphi}^{Q}f)\chi_{\Sigma}\|_{p,\mathbb{R}^{4}} = \|(S_{\varphi}^{Q}f)\chi_{\Sigma\cap B_{r}}\|_{p,\mathbb{R}^{4}} + \|(S_{\varphi}^{Q}f)\chi_{\Sigma \cap B_{r}^{c}}\|_{p,\mathbb{R}^{4}}.
\end{equation*}
However, by H$\ddot{o}$lder's inequality, we get for every $\alpha > 0$,\ \\ \ \\
$\bigg(\displaystyle\int_{\mathbb{R}^{2}}\int_{\mathbb{R}^{2}} |(S_{\varphi}^{Q}f)(\xi,b)|^{p}\chi_{B_{r}\cap\Sigma}(\xi,b) d\mu_{4}(\xi,b)\bigg)^{\frac{1}{p}}$
\begin{eqnarray*}
& = & \bigg(\displaystyle\int_{\mathbb{R}^{2}}\int_{\mathbb{R}^{2}} |(S_{\varphi}^{Q}f)(\xi,b)|^{\frac{p^{2}}{p+1}} |(S_{\varphi}^{Q}f)(\xi,b)|^{\frac{p}{p+1}}\chi_{B_{r}\cap\Sigma}(\xi,b) d\mu_{4}(\xi,b)\bigg)^{\frac{1}{p}}\\
& \leq & \|(S_{\varphi}^{Q}f)\|_{\infty,\mathbb{R}^{4}}^{\frac{p}{p+1}}\bigg(\displaystyle\int_{\mathbb{R}^{2}}\int_{\mathbb{R}^{2}}  |(S_{\varphi}^{Q}f)(\xi,b)\chi_{B_{r}}(\xi,b)|^{\frac{p}{p+1}} |\chi_{\Sigma}(\xi,b)|^{\frac{1}{p+1}} d\mu_{4}(\xi,b)\bigg)^{\frac{1}{p}}\\
& \leq & (\|f\|_{2,\mathbb{R}^{2}} \|\varphi\|_{2,\mathbb{R}^{2}})^{\frac{p}{p+1}} (\mu_{\mathbb{R}^{4}}(\Sigma))^{\frac{1}{p(p+1)}} \|(S_{\varphi}^{Q}f)\chi_{B_{r}} \|_{1,\mathbb{R}^{4}}^{\frac{1}{p+1}}\\
& \leq & (\mu_{\mathbb{R}^{4}}(\Sigma))^{\frac{1}{p(p+1)}} \||(\xi,b)|^{-\alpha}(S_{\varphi}^{Q}f) \|_{2,\mathbb{R}^{4}}^{\frac{1}{p+1}} \||(\xi,b)|^{\alpha}\chi_{B_{r}} \|_{2,\mathbb{R}^{4}}^{\frac{1}{p+1}}.
\end{eqnarray*}
On the other hand,\ \\ \ \\
\begin{eqnarray*}
\bigg(\displaystyle\int_{\mathbb{R}^{2}}\int_{\mathbb{R}^{2}} |(\xi,b)|^{2\alpha}\chi_{B_{r}}(\xi,b) d\mu_{4}(\xi,b)\bigg)^{\frac{1}{2(p+1)}}
& = & \bigg(\dfrac{1}{(2\pi)^{2}}\displaystyle\iint_{B_{r}} |(\xi,b)|^{2\alpha} d\xi\ db \bigg)^{\frac{1}{2(p+1)}} \\
& = & \bigg(\dfrac{2\pi^{2}}{(2\pi)^{2} \Gamma(2)} \displaystyle\int_{0}^{r} t^{2\alpha+3} dt\bigg)^{\frac{1}{2(p+1)}}\\
& = & \bigg(\dfrac{1}{2\sqrt{\alpha+2}}\bigg)^{\frac{1}{p+1}} r^{\frac{\alpha+2}{p+1}},
\end{eqnarray*}
so
\begin{equation*}
\|(S_{\varphi}^{Q}f)\chi_{\Sigma\cap B_{r}}\|_{p,\mathbb{R}^{4}} \leq \bigg(\dfrac{1}{2\sqrt{\alpha+2}}\bigg)^{\frac{1}{p+1}} (\mu_{\mathbb{R}^{4}}(\Sigma))^{\frac{1}{p(p+1)}}  \||(\xi,b)|^{-\alpha} (S_{\varphi}^{Q}f) \|_{2,\mathbb{R}^{4}}^{\frac{1}{p+1}}\ r^{\frac{\alpha+2}{p+1}}  .
\end{equation*}
On the other hand, and again by H$\ddot{o}$lder's inequality and relation (\ref{LINFTY}), we deduce that\ \\ \ \\
$\bigg(\displaystyle\int_{\mathbb{R}^{2}}\int_{\mathbb{R}^{2}} |(S_{\varphi}^{Q}f)(\xi,b)|^{p}\chi_{B_{r}^{c}\cap\Sigma}(\xi,b) d\mu_{4}(\xi,b)\bigg)^{\frac{1}{p}}$
\begin{eqnarray*}
& = & \bigg(\displaystyle\int_{\mathbb{R}^{2}}\int_{\mathbb{R}^{2}} |(S_{\varphi}^{Q}f)(\xi,b)|^{\frac{p(p-1)}{p+1}} |(S_{\varphi}^{Q}f)(\xi,b)|^{\frac{2p}{p+1}}\chi_{B_{r}^{c}\cap\Sigma}(\xi,b) d\mu_{4}(\xi,b)\bigg)^{\frac{1}{p}}\\
& \leq & \|(S_{\varphi}^{Q}f)\|_{\infty,\mathbb{R}^{4}}^{\frac{p-1}{p+1}}\bigg(\displaystyle\int_{\mathbb{R}^{2}}\int_{\mathbb{R}^{2}}  (|(S_{\varphi}^{Q}f)(\xi,b)|^{2}\chi_{B_{r}^{c}}(\xi,b))^{\frac{p}{p+1}} |\chi_{\Sigma}(\xi,b)|^{\frac{1}{p+1}} d\mu_{4}(\xi,b)\bigg)^{\frac{1}{p}}\\
& \leq & (\|f\|_{2,\mathbb{R}^{2}} \|\varphi\|_{2,\mathbb{R}^{2}})^{\frac{p-1}{p+1}} (\mu_{\mathbb{R}^{4}}(\Sigma))^{\frac{1}{p(p+1)}} \||(S_{\varphi}^{Q}f)|^{2}\chi_{B_{r}^{c}} \|_{1,\mathbb{R}^{4}}^{\frac{1}{p+1}}\\
& \leq & (\mu_{\mathbb{R}^{4}}(\Sigma))^{\frac{1}{p(p+1)}} \||(\xi,b)|^{-\alpha}(S_{\varphi}^{Q}f) \|_{2,\mathbb{R}^{4}}^{\frac{2}{p+1}} \||(\xi,b)|^{2\alpha}\chi_{B_{r}^{c}} \|_{\infty,\mathbb{R}^{4}}^{\frac{1}{p+1}}\\
& \leq & (\mu_{\mathbb{R}^{4}}(\Sigma))^{\frac{1}{p(p+1)}} \||(\xi,b)|^{-\alpha}(S_{\varphi}^{Q}f) \|_{2,\mathbb{R}^{4}}^{\frac{2}{p+1}}\ r^{-\frac{2\alpha}{p+1}}.
\end{eqnarray*}
Hence,
\begin{eqnarray*}
\|(S_{\varphi}^{Q}f)\chi_{\Sigma}\|_{p,\mathbb{R}^{4}} \leq (\mu_{\mathbb{R}^{4}}(\Sigma))^{\frac{1}{p(p+1)}} \||(\xi,b)|^{-\alpha} (S_{\varphi}^{Q}f)\|_{2,\mathbb{R}^{4}}^{\frac{1}{p+1}} \bigg[\bigg(\dfrac{1}{2\sqrt{\alpha+2}}\bigg)^{\frac{1}{p+1}}\ r^{\frac{\alpha+2}{p+1}} + \||(\xi,b)|^{-\alpha} (S_{\varphi}^{Q}f)\|_{2,\mathbb{R}^{4}}^{\frac{1}{p+1}} \ r^{-\frac{2\alpha}{p+1}}\bigg].
\end{eqnarray*}
In particular, the inequality holds for
\begin{equation*}
r_{0} = \bigg[2\sqrt{\alpha+2}\ \bigg(\dfrac{2\alpha}{\alpha+2}\bigg)^{p+1} \||(\xi,b)|^{-\alpha}(S_{\varphi}^{Q}f)\|_{2,\mathbb{R}^{2}} \bigg]^{\frac{1}{3\alpha+2}},
\end{equation*}
and therefore,
\begin{eqnarray*}
\|(S_{\varphi}^{Q}f)\chi_{\Sigma}\|_{p,\mathbb{R}^{4}} \leq (\mu_{\mathbb{R}^{4}}(\Sigma))^{\frac{1}{p(p+1)}} \bigg[ \bigg(\dfrac{2\alpha}{\alpha+2}\bigg)^{\frac{\alpha+2}{3\alpha+2}} + \bigg(\dfrac{\alpha+2}{2\alpha}\bigg)^{\frac{2\alpha}{3\alpha+2}}\bigg]\\
\bigg(\dfrac{1}{2\sqrt{\alpha+2}}\bigg)^{\frac{2\alpha}{(3\alpha+2)(p+1)}} \||(\xi,b)|^{-\alpha}(S_{\varphi}^{Q}f)\|_{2,\mathbb{R}^{2}}^{\frac{4(\alpha+1)}{(3\alpha+2)(p+1)}}.
\end{eqnarray*}
Suppose that $f \neq 0$, $g = \dfrac{f}{\|f\|_{2,\mathbb{R}^{2}}}$ and $\psi = \dfrac{\varphi}{\|\varphi\|_{2,\mathbb{R}^{2}}}$, we have $\|g\|_{2,\mathbb{R}^{2}} = \|\psi\|_{2,\mathbb{R}^{2}} = 1$ and \\ \\
$\|(S_{\psi}^{Q}g)\chi_{\Sigma}\|_{p,\mathbb{R}^{4}} = $
\begin{eqnarray*}
(\mu_{\mathbb{R}^{4}}(\Sigma))^{\frac{1}{p(p+1)}} \bigg[ \bigg(\dfrac{2\alpha}{\alpha+2}\bigg)^{\frac{\alpha+2}{3\alpha+2}} + \bigg(\dfrac{\alpha+2}{2\alpha}\bigg)^{\frac{2\alpha}{3\alpha+2}}\bigg]
\bigg(\frac{1}{2\sqrt{\alpha+2}}\bigg)^{\frac{2\alpha}{(3\alpha+2)(p+1)}} \||(\xi,b)|^{-\alpha}(S_{\psi}^{Q}g)\|_{2,\mathbb{R}^{2}}^{\frac{4(\alpha+1)}{(3\alpha+2)(p+1)}}.
\end{eqnarray*}
We have
\begin{equation*}
(S_{\psi}^{Q}g) = \dfrac{(S_{\varphi}^{Q}f)}{\|f\|_{2,\mathbb{R}^{2}}\|\varphi\|_{2,\mathbb{R}^{2}}} ,
\end{equation*}
and consequently
\begin{eqnarray*}
\|(S_{\varphi}^{Q}f)\chi_{\Sigma}\|_{p,\mathbb{R}^{4}} \leq M_{\alpha,p}\  (\|f\|_{2,\mathbb{R}^{2}}\|\varphi\|_{2,\mathbb{R}^{2}})^{1-\frac{4(\alpha+1)}{(3\alpha+2)(p+1)}}(\mu_{\mathbb{R}^{4}}(\Sigma))^{\frac{1}{p(p+1)}}  \||(\xi,b)|^{-\alpha}(S_{\varphi}^{Q}f)\|_{2,\mathbb{R}^{2}}^{\frac{4(\alpha+1)}{(3\alpha+2)(p+1)}},
\end{eqnarray*}
we note
\begin{equation*}
M_{\alpha,p} = \bigg(\dfrac{3\alpha+2}{2\alpha}\bigg)\bigg(\dfrac{2\alpha}{\alpha+2}\bigg)^{\frac{\alpha+2}{3\alpha+2}}
\bigg(\dfrac{1}{2\sqrt{\alpha+2}}\bigg)^{\frac{2\alpha}{(3\alpha+2)(p+1)}}.
\end{equation*}
\end{proof}
\begin{definition} A function $F$ in $L^{2}(\mathbb{R}^{2}\times\mathbb{R}^{2},\mathbb{H})$ is said to be $\alpha$-concentration on a measurable set $\Sigma \subset \mathbb{R}^{2}\times\mathbb{R}^{2}$, where $\Sigma^{c} = \mathbb{R}^{2}\times\mathbb{R}^{2}\setminus \Sigma$, if
\begin{equation}\label{concentrated}
\bigg(\displaystyle\int_{\Sigma^{c}} |F(x)|^{2}d\mu_{4}(x)\bigg)^{\frac{1}{2}} \leq \alpha \|F\|_{2,\mathbb{R}^{4}}.
\end{equation}
If $0\leq \alpha \leq \frac{1}{2}$, then the most of energy is concentrated on $\Sigma$, and $\Sigma$ can be called the essential support of $F$.\\ \ \\
If  $\alpha = 0$, then $\Sigma$ contain the support of $F$.
\end{definition}
\begin{theorem} (Donoho-Stark for $S_{\varphi}^{Q}$)\\
Let $\varphi$ be a non zero admissible quaternion window satisfies the assumption of theorem \ref{conditions} , and for every $f$ in $L^{2}(\mathbb{R}^{2}\times\mathbb{R}^{2},\mathbb{H})$ such that $f\neq 0$. Let $\Sigma$ a measurable set of $\mathbb{R}^{2}\times\mathbb{R}^{2}$ and $\alpha \geq 0$. If $(S_{\varphi}^{Q}f)$ is $\alpha$- concentrated on $\Sigma$, hence we have
\begin{equation}\label{donohostark}
\mu_{\mathbb{R}^{4}}(\Sigma) \geq \dfrac{(1-\alpha^{2})\ C_{\varphi}}{\|\varphi\|_{2,\mathbb{R}^{2}}^{2}}.
\end{equation}
\end{theorem}
\begin{proof} Let $\Sigma \subset \mathbb{R}^{2}\times\mathbb{R}^{2}$ a measurable subset and $\chi_{\Sigma^{c}} = 1 - \chi_{\Sigma}$. We may the write
\begin{eqnarray}\label{hypo1}
\|(S_{\varphi}^{Q}f)\|_{2,\mathbb{R}^{4}}^{2} = \|\chi_{\Sigma}(S_{\varphi}^{Q}f)\|_{2,\mathbb{R}^{4}}^{2} + \|\chi_{\Sigma^{c}}(S_{\varphi}^{Q}f)\|_{2,\mathbb{R}^{4}}^{2}
\end{eqnarray}
by relations (\ref{plancherelformula}) and (\ref{concentrated}), we have
\begin{eqnarray}\label{hypo2}
\|\chi_{\Sigma^{c}}(S_{\varphi}^{Q}f)\|_{2,\mathbb{R}^{4}}^{2} \leq \alpha^{2} C_{\varphi} \|f\|_{2,\mathbb{R}^{2}}^{2}
\end{eqnarray}
using (\ref{hypo1}) and (\ref{hypo2}), we obtain
\begin{eqnarray} \label{hyp3}
(1 - \alpha^{2}) C_{\varphi} \|f\|_{2,\mathbb{R}^{2}}^{2} \leq \|\chi_{\Sigma}(S_{\varphi}^{Q}f)\|_{2,\mathbb{R}^{4}}^{2}
\end{eqnarray}
consequently by using (\ref{LINFTY}), we deduce
\begin{eqnarray*}
(1 - \alpha^{2}) C_{\varphi} \|f\|_{2,\mathbb{R}^{2}}^{2} \leq \mu_{\mathbb{R}^{4}}(\Sigma) \|(S_{\varphi}^{Q}f)\|_{\infty,\mathbb{R}^{4}}^{2} \leq \mu_{\mathbb{R}^{4}}(\Sigma) \|f\|_{2,\mathbb{R}^{2}}^{2}\|\varphi\|_{2,\mathbb{R}^{2}}^{2}
\end{eqnarray*}
we may simplify by $\|f\|_{2,\mathbb{R}^{2}}^{2}$ to obtain the desired result (\ref{donohostark}).
\end{proof}
\begin{theorem} (Lieb uncertainty principle for $S_{\varphi}^{Q}$)\\
Let $\varphi$ be a non zero admissible quaternion window satisfies the assumption of theorem \ref{conditions} , and for every $f$ in $L^{2}(\mathbb{R}^{2}\times\mathbb{R}^{2},\mathbb{H})$ such that $f\neq 0$. Let $\Sigma$ a measurable set of $\mathbb{R}^{2}\times\mathbb{R}^{2}$ and $\alpha \geq 0$. If $(S_{\varphi}^{Q}f)$ is $\alpha$- concentrated on $\Sigma$, hence for evry $p> 2$ we have
\begin{equation}
\mu_{\mathbb{R}^{4}}(\Sigma) \geq \dfrac{C_{\varphi} \ (1 - \alpha^{2})^{\frac{p}{p-2}}}{\|\varphi\|_{2,\mathbb{R}^{2}}^{2} } .
\end{equation}
\end{theorem}
\begin{proof} We have \\ \ \\
$\displaystyle\int_{\mathbb{R}^{2}}\int_{\mathbb{R}^{2}}\chi_{\Sigma}(\xi,b)|(S_{\varphi}^{Q}f)(\xi,b)|^{2} d\mu_{4}(\xi,b)$
\begin{eqnarray*}
\leq \bigg(\displaystyle\int_{\mathbb{R}^{2}}\int_{\mathbb{R}^{2}} (\chi_{\Sigma}(\xi,b))^{\frac{p}{p-2}} d\mu_{4}(\xi,b)\bigg)^{\frac{p-2}{p}} \bigg(\displaystyle\int_{\mathbb{R}^{2}}\int_{\mathbb{R}^{2}}(|(S_{\varphi}^{Q}f)(\xi,b)|^{2})^{\frac{p}{2}} d\mu_{4}(\xi,b)\bigg)^{\frac{2}{p}}.
\end{eqnarray*}
On the other hand, and again by theorem \ref{lieb} and relation (\ref{hyp3}), we deduce that
\begin{equation*}
(1 - \alpha^{2})\ C_{\varphi} \|f\|_{2,\mathbb{R}^{2}}^{2} \leq (\mu_{\mathbb{R}^{4}}(\Sigma))^{\frac{p-2}{p}}\ C_{\varphi}^{\frac{2}{p}}\ \|f\|_{2,\mathbb{R}^{2}}^{2} \|\varphi\|_{2,\mathbb{R}^{2}}^{2-\frac{4}{p}}
\end{equation*}
we may simplify by $C_{\varphi}\ \|f\|_{2,\mathbb{R}^{2}}^{2}$ to obtain
\begin{equation*}
(1 - \alpha^{2}) \leq (\mu_{\mathbb{R}^{4}}(\Sigma))^{\frac{p-2}{p}}\ C_{\varphi}^{\frac{2-p}{p}}  \|\varphi\|_{2,\mathbb{R}^{2}}^{2(\frac{p-2}{p})}
\end{equation*}
and consequently we get
\begin{equation*}
\mu_{\mathbb{R}^{4}}(\Sigma) \geq \dfrac{C_{\varphi} \ (1 - \alpha^{2})^{\frac{p}{p-2}} }{\|\varphi\|_{2,\mathbb{R}^{2}}^{2} } .
\end{equation*}
\end{proof}

\end{document}